\documentclass{amsart}
\usepackage{amssymb}
\usepackage{setspace}
\newtheorem{theorem}{Theorem}[section]
\newtheorem*{theorem*}{Theorem}
\newtheorem{lemma}[theorem]{Lemma}
\newtheorem{proposition}[theorem]{Proposition}
\newtheorem{corollary}[theorem]{Corollary}

\theoremstyle{remark}
\newtheorem*{remark*}{Remark}

\newtheorem*{example*}{Example}
\newcommand{\partialgamma}{\partial^\gamma\!}
\newcommand{\Dgamma}{D^\gamma\!}
\newcommand{\Dalpha}{D^\alpha\!}
\newcommand{\Dbeta}{D^\beta\!}

\begin{document}
\title[ABC theorems for non-Archimedean entire functions]
{Generalized ABC theorems for non-Archimedean entire functions of several
variables in arbitrary characteristic}
\author{William Cherry}
\thanks{Partial financial support for this research was provided
by a Faculty Research Grant from the University of North Texas and
by the United States National Security Agency under
Grant Number H98230-07-1-0037.
The United States Government is authorized to reproduce and
distribute reprints not­withstanding any copyright notation herein.}
\address{Department of Mathematics\\University of North Texas\\
P.O. Box 311430, Denton, TX  76203-1430\\USA}
\email{wcherry@unt.edu}
\author{Cristina Toropu}
\address{Department of Mathematics and Statistics\\ MSC03 2150\\
1 University of New Mexico
\\ Albuquerque, New Mexico, 87131-0001\\
USA}
\email{ctoropu@hotmail.com}
\date{July 29, 2008}
\begin{abstract}
We prove  generalized ABC theorems for vanishing sums of non-Archimedean
entire functions of several variables in arbitrary characteristic.
\end{abstract}
\subjclass[2000]{11D04, 11D88, 32P05}
\keywords{ABC theorems, non-Archimedean entire functions, several variables,
Hasse derivatives, vanishing sums, Mason's theorem, generalized Wronskians}
\maketitle
\section{Introduction}
The well-known ABC Theorem for polynomials,
first proved by Stothers \cite{Stothers}
and often called ``Mason's Theorem'' \cite{MasonDiophEqn},
says that if $a+b=c,$
where $a$ and $b$ are relatively prime univariate polynomials
with at least one of the derivatives $a'$ or $b'$ not identically zero, then
$$
    \max\{\deg a, \deg b, \deg c\} \le \deg R(abc)-1.
$$
Here if $f$ is a polynomial, we use $R(f)$ to denote
$f/\gcd(f,f').$
In characteristic zero, $R(f)$
is simply the degree of the square free part of $f,$
also called the degree of the radical of $f,$  and is the
number of distinct zeros of $f$ in an algebraically closed field
containing the coefficients of $f.$   We note that the above theorem
immediately extends to polynomials of several variables by replacing
the ordinary derivative with a partial derivative.

The existence of an appropriately analogous inequality for
$c=a+b$ with $a,$ $b,$
and $c$ relatively prime integers is the famous ABC Conjecture
of Masser and Oesterl\'e (see \cite{Oesterle}), which states that
for each $\varepsilon>0,$ there exists a constant $C(\varepsilon)$
such that
$$
    \max\{|a|,|b|,|c|\} \le C(\varepsilon)S(abc)^{1+\varepsilon},
$$
where we use $S(abc)$ to denote the square free part of $abc.$ The
ABC conjecture for integers has spectacular consequences in number
theory -- see \textit{e.g.,} \cite{Goldfeld} and \cite{Lang}. To
date, in the case of integers, the best proven upper bounds on
$\max\{|a|,|b|,|c|\}$ in terms of $S(abc)$ are super-polynomial in
$S(abc)$ -- see \textit{e.g.,} \cite{StewartYu}.

The ABC theorem for polynomials
has been generalized in a variety of directions,
including: to sums in one-dimensional function fields by
Mason \cite{MasonNormForm}, by Voloch \cite{Voloch}
and by Brownawell and Masser \cite{BrownawellMasser}, to sums of
pairwise relatively prime polynomials
of several variables by Shapiro and Sparer \cite{ShapiroSparer},
to sums in higher-dimensional function fields by Hsia and Wang
\cite{HsiaWang}, and to quantum deformations of polynomials
by Vaserstein \cite{Vaserstein}. Motivated by the analogy between
Diophantine approximation and Nevanlinna theory \cite{Vojta},
the ABC theorem has also been proven for complex
entire functions by Van Frankenhuysen \cite{vanFrankenhuysenThesis},
\cite{vanFrankenhuysenPreprint} and for $p$-adic entire functions
by Hu and Yang \cite{HuYangABC}.

In a recent article,
An and Manh \cite{AnManh} gave an ABC type theorem for $p$-adic
entire functions of several variables $c=a+b,$ but under some rather
restrictive hypotheses, including the assumption that $a,$ $b,$
and $c$ have no common zeros, which in several variables is a much
stronger assumption than simply supposing that $a,$ $b,$ and $c$
are pairwise relatively prime in the ring of entire functions.
The purpose of this article is to prove general ABC theorems for sums
$$
    f_n=f_0+\dots+f_{n-1}
$$
of non-Archimedean entire functions on affine $m$-space $\mathbf{A}^m$
in arbitrary characteristic analogous to the existing theorems
for several variable polynomials.
We do this for two reasons.  First, we illustrate
that if the existing polynomial proofs for ABC theorems
are correctly interpreted, then they provide immediate proofs
for the analogous statements for non-Archimedean entire functions
without the need for any fundamentally new ideas or for additional technical
assumptions as in \cite{AnManh}, with the exception of one slight
subtlety discussed just before Corollary~\ref{abcsf}.
Second, Cherry and Ye \cite{CherryYe}
developed a several variable non-Archimedean Nevanlinna theory in a
form that was intended to be convenient to use.  However in applications,
it is often convenient to work with, or at least to think in terms of,
truncated counting functions, which
were not discussed in \cite{CherryYe}.
Thus, our second purpose is to illustrate how to define and work
with truncated counting functions in several variables and in positive
characteristic; this is well-known to the experts, but we thought it
helpful to illustrate here for the novice's benefit.

The plan of this paper is as follows. In section~\ref{prelim} we set
up some notation and recall some basic facts we will need.  We also
prove the most basic form of the ABC theorem in
section~\ref{prelim}, so readers only interested in the basic idea
do not need to read past section~\ref{prelim}. In
section~\ref{Hasse}, we recall the notion of Hasse derivative and
generalized Wronskians necessary to work in positive characteristic
and in several variables. In section~\ref{truncated}, we discuss how
to define truncated counting functions in positive characteristic.
We recall some linear algebra from Brownawell and Masser
\cite{BrownawellMasser} in section~\ref{linalg}. Finally in
section~\ref{ABC}, we derive our general ABC theorems and indicate
how various ABC theorems in the literature can be derived as
corollaries. Our method of proof is essentially that of Hu and Yang
\cite{HuYangBB}, who adapted the argument of Brownawell and Masser
\cite{BrownawellMasser} to the context of non-Archimedean entire
functions of one variable. Our presentation in section~\ref{ABC} was
also influenced by the recent work of De~Bondt \cite{deBondt}, who
formulated generalized ABC theorems for complex polynomials in
several variables from which the other various versions in the
literature can be derived.

\section{Preliminaries, Notation and Warm-up}\label{prelim}

We will find it convenient to use some Nevanlinna notation.
We will use \cite{CherryYe} as our basic reference and mostly follow
the notation there, although here we will not assume characteristic zero,
as was done in \cite{CherryYe}.

Throughout, $\mathbf{F}$ will denote an algebraically closed field
complete with respect to a non-Archimedean absolute value $|~|.$
We make no assumption about the characteristic of $\mathbf{F}.$
Let $\mathbf{F}^\times$ denote $\mathbf{F}\setminus\{0\},$
and let $|\mathbf{F}^\times|$ be the subset of the positive real numbers
defined by
$$
    |\mathbf{F}^\times| = \{ |a| : a \in \mathbf{F}^\times\}.
$$

Let $\mathbf{F}^m$ denote the $m$-th Cartesian product of $F,$ which
is the set of $\mathbf{F}$-points of affine $m$-space $\mathbf{A}^m.$
By an \textbf{entire function} on $\mathbf{A}^m$ or $\mathbf{F}^m,$
we mean a formal power series in $m$ variables with coefficients
in $\mathbf{F}$ and with infinite radius of convergence.
We will use $\mathcal{E}_m$ to denote the ring of entire functions
on $\mathbf{A}^m.$

If $z_1,\dots,z_m$ are $\mathbf{F}$-valued variables, we use $z$
to refer collectively to the $m$-tuple $(z_1,\dots,z_m).$
When convenient, we will use multi-index notation.
If $\gamma=(\gamma_1,\dots,\gamma_m)$ is a multi-index, \textit{i.e.,}
an $m$-tuple of non-negative integers, then by definition
$$
    z^\gamma=z_1^{\gamma_1}\cdots z_m^{\gamma_m},
    \qquad
    |\gamma|=\gamma_1+\dots+\gamma_m, \qquad\textnormal{and}\qquad
    \partialgamma f = \frac{\partial^{|\gamma|}f}{\partial z^\gamma}.
$$
Similarly, if $\mathbf{r}=(r_1,\dots,r_m)$ is an $m$-tuple of non-negative
real numbers, we define
$$
    \mathbf{r}^\gamma = r_1^{\gamma_1}\cdots r_m^{\gamma_m}.
$$
We can therefore write an entire function $f$ in $\mathcal{E}_m$ as
$$
    f(z)=\sum_\gamma a_\gamma z^\gamma
$$
where $a_\gamma$ are in $\mathbf{F},$ and for all $m$-tuples of
non-negative real numbers $\mathbf{r},$
$$
    \lim_{|\gamma|\to\infty}|a_\gamma|\mathbf{r}^\gamma=0.
$$

We recall that to each $m$-tuple of non-negative real numbers
$\mathbf{r}=(r_1,\dots,r_m),$ we can associate a non-Archimedean absolute
value $|~|_{\mathbf{r}}$ on the ring $\mathcal{E}_m$ by defining
$$
    |f|_{\mathbf{r}} = \sup_\gamma |a_\gamma|{\mathbf{r}}^\gamma,
$$
where as above $f$ in $\mathcal{E}_m$ is given by the power series expansion
$$
    f(z)=\sum_\gamma a_\gamma z^\gamma.
$$
The non-trivial thing that needs to be checked is that if
$f$ and $g$ are two elements of $\mathcal{E}_m,$
then $|fg|_{\mathbf{r}}=|f|_{\mathbf{r}}|g|_{\mathbf{r}}.$  When all the
$r_j=1$ this is worked out,
for instance, in \cite[\S5.1.2]{BGR}.  In general, by extending the
field $\mathbf{F}$ if necessary, we may assume
the $r_j$ are elements of $|\mathbf{F}^\times|$ and then reduce
to the case when all the $r_j$ are $1$ by an affine rescaling of
the variables.

For our purposes, we only need consider those $\mathbf{r}$ for which
all the $r_j$ are equal, or in other words $m$-tuples of the
form $\mathbf{r}=(r,\dots,r).$  We will denote the associated absolute
value on $\mathcal{E}_m$ by $|~|_r.$ Clearly,

\begin{proposition}\label{frnondec}
If $f$ is in $\mathcal{E}_m,$ then $|f|_r$
is a non-decreasing function of $r.$
\end{proposition}

If $f$ happens to be a polynomial of degree $d,$ then we easily see
that as $r\to\infty,$
$$
    \log|f|_r = d\log r + O(1).
$$
Thus, in our ABC theorems for entire functions, $\log|f|_r$
will play the roll played by the degree in the case of polynomials
in the left-hand side of the inequalities.

We will make use of the following observation on several occasions:

\begin{corollary}\label{factorsmaller}
If $f,$ $g$ and $h$ are elements of $\mathcal{E}_m$ such that
$f=gh$ and if $r_0>0,$ then
for all $r\ge r_0,$
$$
    \log|g|_r\le\log|f|_r+O(1).
$$
\end{corollary}

\begin{proof}
By the multiplicativity of $|~|_r$ and Proposition~\ref{frnondec},
$$
    \log|f|_r=\log|g|_r+\log|h|_r \ge \log|g|_r+\log|h|_{r_0},
$$
which gives the first inequality.
\end{proof}

We recall the elementary:

\begin{lemma}[Logarithmic Derivative Lemma]\label{LDL}
Let $f$ be an entire function in $\mathcal{E}_m$ and let $\gamma$
be a multi-index. Then,
$$
    |\partialgamma f|_r \le \frac{|f|_r}{r^{|\gamma|}}.
$$
\end{lemma}

\begin{proof}
Differentiate the power series defining $f$ and use the fact that
$|k|\le1$ for any integer $k.$
\end{proof}

We need to make use of some ring-theoretic properties of
$\mathcal{E}_m.$  The reader can see \cite{CherryGCD}
for a detailed treatment.

As usual, we will call an element $P$ of $\mathcal{E}_m$
\textbf{irreducible} if whenever we write $P=fg$ with $f$ and $g$ in
$\mathcal{E}_m,$ we must have that at least one of $f$ or $g$ is a
unit in $\mathcal{E}_m.$ As is well-known, the only units in
$\mathcal{E}_m$ are the nonzero constant functions -- see,
\textit{e.g.,} \cite[Cor.~2.4]{CherryYe}.

\begin{proposition}\label{derivirred}
Let $P$ be an irreducible element of $\mathcal{E}_m$ and let $j$
be an integer between $1$ and $m.$  If $P$ divides
$\partial P/\partial z_j,$ then $\partial P/\partial z_j\equiv0.$
\end{proposition}

\begin{proof}
Suppose
$$
    \frac{\partial P}{\partial z_j}=Pg
$$
for some $g$ in $\mathcal{E}_m.$  Then,
$$
    \left|\frac{\partial P}{\partial z_j}\right|_r=|P|_r|g|_r.
$$
From Lemma~\ref{LDL}, $|g|_r\le 1/r,$ and so by Proposition~\ref{frnondec},
$|g|_r\equiv0.$
\end{proof}

Suppose $f$ and $g$ are non-constant
elements of $\mathcal{E}_m$ such that $g$ divides $f$ in $\mathcal{E}_m.$
Then $f$ and $g$ can also be considered as analytic
functions on the ``closed'' ball of radius $r,$ \textit{i.e.,}
$$
    \mathbf{B}^m(r)=\{(z_1,\dots,z_m)\in\mathbf{F}^m : \max |z_j|\le r\}.
$$
For large $r,$
$g$ will have zeros inside the ball (again see \cite[Cor.~2.4]{CherryYe}),
and hence will not be a unit in the ring of analytic functions
on the ball $\mathbf{B}^m(r).$
These rings are Tate algebras when $r\in|\mathbf{F}^\times|,$
and hence factorial
\cite[\S5.2.6, Th.~1]{BGR}. Thus,
$\mathcal{E}_m$ is a subring of a factorial ring in which
$g$ is not a unit, and hence some power of $g$ will not divide $f$
in $\mathcal{E}_m.$  Therefore, one can speak of the multiplicity with
which an entire function $g$ divides another entire function $f.$
Although $\mathcal{E}_m$ itself is not factorial,
the notion of ``greatest common divisor'' does make sense
in $\mathcal{E}_m;$
see \cite{CherryGCD} or the appendix of \cite{CherryYe}.
Of course, greatest common
divisors are only defined up to units, hence
multiplicative constants.
Given two entire functions $f_1$ and $f_2,$
when we write something like $g=\gcd(f_1,f_2),$ we mean pick any function
$g$ which is a greatest common divisor of $f_1$ and $f_2.$ Hence
$g$ is only well-defined up to a choice of multiplicative constant.

Since we have greatest common divisors, we can define
a good notion of the ``radical'' or the ``square free part''
of an entire function, at least in characteristic zero.
The definition we give here will not be the square free part
in positive characteristic, but will be the suitable thing to put
on the right hand side in our basic ABC theorem.  We will discuss
the existence of the square free part of an analytic function in positive
characteristic in a later section.

\begin{proposition}\label{radical}
Let $f$ be an entire function in $\mathcal{E}_m.$
For each $j$ from $1$ to $m,$ denote by
$$
    g_j = \gcd\left(f,\frac{\partial f}{\partial z_j}\right)
    \qquad\textnormal{and}\qquad
    h_j = \frac{f}{g_j}.
$$
Let $R(f)$ be the least common multiple of the $h_j.$
Then,
\begin{enumerate}
\item[(i)] $R(f)$ divides $f;$
\item[(ii)] for any non-constant $g$ in $\mathcal{E}_m,$
$g^2$ does not divide $R(f);$ and
\item[(iii)] if $P$ is an irreducible element
of $\mathcal{E}_m$ that divides $f,$ then $P$ also divides $R(f)$
if and only if the multiplicity to which $P$ divides $f$ is
not divisible by the characteristic of $\mathbf{F}.$
\end{enumerate}
\end{proposition}

We will call $R(f)$ as defined in Proposition~\ref{radical}
\textbf{the radical} of $f.$

\begin{proof}
For~(i), note that because each of the $h_j$ divides $f,$  it is
clear that their least common multiple $R(f)$ also divides $f.$

To show~(ii), suppose $g$ is a non-constant element of $\mathcal{E}_m$
such that $g^2$ divides $R(f).$  Then, $g^2$ must also
divide $f$ since $R(f)$ divides $f.$  Let $s\ge2$ be the largest
integer such that $g^s$ divides $f.$  Then,
$f=g^s\tilde f,$ where $\tilde f$ is an element of $\mathcal{E}_m$
not divisible by $g.$  Because
$$
    \frac{\partial f}{\partial z_j} =
    \tilde f s g^{s-1} \frac{\partial g}{\partial z_j}
    + g^s \frac{\partial \tilde f}{\partial z_j},
$$
we see that $g^{s-1}$ divides $g_j$ for all $j,$ and hence $g^2$
does not divide $h_j$ for any $j.$ Thus, $g^2$ cannot divide $R(f).$

To show~(iii), let $P$ be a non-constant irreducible element of $\mathcal{E}_m$
that divides $f.$  Let $s$ be the largest integer such that
$P^s$ divides $f.$  Then, $f=P^s\tilde f$ with $\tilde f$
relatively prime to $P.$ If $\partial P/\partial z_j\equiv0$ for all $j,$
then because we have assumed that $P$ is non-constant, it must be that
$P$ is a $p$-th power and that
$\mathbf{F}$ has positive characteristic $p.$  But then, $P$ would not
be irreducible, and so there must exist some $j$ such that
$\partial P/\partial z_j\not\equiv0.$ Because
$$
    \frac{\partial f}{\partial z_j} = sP^{s-1}
    \frac{\partial P}{\partial z_j} \tilde f + P^s\frac{\partial \tilde f}
    {\partial z_j},
$$
we conclude from Proposition~\ref{derivirred} that $P^s$
divides $\partial f/\partial z_j$ if and only if $s=0$ in $\mathbf{F}.$
\end{proof}

We will now state and prove the most
basic version of an ABC theorem for non-Archimedean entire functions
of several variables.  We feel that discussing this
basic case here will help the reader see the main ideas behind what
we will do in the sequel.

\begin{theorem}[Basic ABC Theorem]\label{basic}
Let $f_2=f_0+f_1$ be entire functions such that $f_0$ and $f_1$
are relatively prime in $\mathcal{E}_m.$  If $\mathbf{F}$ has
characteristic zero, assume that at least one of $f_0$ or $f_1$
is non-constant.  If $\mathbf{F}$ has positive characteristic $p,$
then assume that at least one of $f_0$ or $f_1$ is not a $p$-th
power in $\mathcal{E}_m.$
Let $r_0>0.$
Then, for $r\ge r_0,$
$$
    \max_{0\le i \le 2} \log |f_i|_r
    \le \log |R(f_0f_1f_2)|_r -\log r + O(1).
$$
\end{theorem}

\begin{proof}
We follow the standard polynomial proof, as given for instance in
\cite{Vaserstein}, \textit{mutatis mutandis.}
Without loss of generality assume that $f_0$ is non-constant and if
$\mathbf{F}$ has positive characteristic $p$ that $f_0$ is not
a $p$-th power in $\mathcal{E}_m.$ This implies there exists a
non-constant irreducible element $P_0$ in $\mathcal{E}_m$ that divides
$f_0$ to a multiplicity $s_0$ not divisible by the characteristic of
$\mathbf{F}.$  Without loss of generality, assume that
$\partial P_0/\partial z_1\not\equiv0.$
Consider the Wronskian determinant,
$$
    W = \det \left ( \begin{array}{cc}
        f_0&f_1\\
\noalign{\vskip 3pt}
        \frac{\displaystyle\partial f_0}{\displaystyle\partial z_1}&
        \frac{\displaystyle\partial f_1}{\displaystyle\partial z_1}
        \end{array}\right)
      = \det \left ( \begin{array}{cc}
        f_0&f_2\\
\noalign{\vskip 3pt}
        \frac{\displaystyle\partial f_0}{\displaystyle\partial z_1}&
        \frac{\displaystyle\partial f_2}{\displaystyle\partial z_1}
        \end{array}\right)
      = \det \left ( \begin{array}{cc}
        f_2&f_1\\
\noalign{\vskip 3pt}
        \frac{\displaystyle\partial f_2}{\displaystyle\partial z_1}&
        \frac{\displaystyle\partial f_1}{\displaystyle\partial z_1}
        \end{array}\right),
$$
where the first equality defines $W$ in $\mathcal{E}_m$
and the second two equalities follow from $f_2=f_0+f_1.$

We first claim that $W\not\equiv0.$  Indeed, if $W\equiv0,$ then
$$
    f_0\frac{\partial f_1}{\partial z_1} =
    f_1\frac{\partial f_0}{\partial z_1}.
$$
Because $P_0^{s_0}$
divides $f_0$ and does not divide $\partial f_0/\partial z_1,$
this would imply that $P_0$ divides $f_1.$  But, $f_0$ and $f_1$ were
assumed relatively prime, and hence $W\not\equiv0.$

Let $F=f_0f_1f_2,$ $G=\gcd(F,\partial F/\partial z_1),$
and $H=F/G.$  Then, by definition $H$ divides $R(f_0f_1f_2),$
and so
$$
    \log|H|_r \le \log|R(f_0f_1f_2)|_r + O(1)
$$
for $r\ge r_0$ by Corollary~\ref{factorsmaller}. We also claim that
$G$ divides $W.$  Indeed, suppose that $P$ is an irreducible element
that divides $G.$  Then $P$ divides $F$ and so it divides one
of $f_i$ and hence exactly one of the $f_i$ since the $f_i$ are
relatively prime.  Thus, suppose that $P$ divides $f_i$ and hence $F$
with multiplicity $s.$ Then, $P^{s-1}$ divides $\partial f_i/\partial z_1$
and hence also $W.$  If $P^s$ also divides $G$ and hence
$\partial F/\partial z_1,$
then either $s$ is divisible by
the characteristic of $\mathbf{F}$ or $\partial P/\partial z_1=0.$
But in either of these cases, $P^s$ also divides $\partial f_i/\partial z_1,$
and so $P^s$ also divides $W.$  Thus, $G$ divides $W$ as claimed.
Again applying Corollary~\ref{factorsmaller}, we see that
for $r\ge r_0,$
$$
    \log|G|_r \le \log |W|_r + O(1).
$$

By Lemma~\ref{LDL},
$$
    \left|f_i\frac{\partial f_j}{\partial z_1}\right|_r \le
    \frac{|f_if_j|_r}{r},
$$
and hence using each of the three determinants defining $W,$
$$
    \log|W|_r\le\log \min\{|f_0f_1|_r,|f_0f_2|_r,|f_1f_2|_r\}-\log r.
$$
Hence,
\begin{align*}
    \log\max|f_i|_r &= \log|f_0|_r+\log|f_1|_r+\log|f_2|_r -
    \log \min_{0\le i < j\le 2} |f_if_j|_r \\
    &=\log|F|_r-\log \min_{0\le i < j\le 2} |f_if_j|_r \\
    &=\log|H|_r+\log|G|_r-\log \min_{0\le i < j\le 2} |f_if_j|_r\\
    &\le \log|R(F)|_r+\log|W|_r-\log \min_{0\le i < j\le 2} |f_if_j|_r
    +O(1)\\
    &\le \log|R(F)|_r-\log r+O(1),
\end{align*}
for $r\ge r_0.$
\end{proof}

We conclude this section with a discussion of counting functions.
For a polynomial in one variable, it is a simple matter to count the
zeros, with or without multiplicity, because they are finite in number.
For a one variable convergent power series, the zeros are discrete, so one can
create a counting function by counting them up to a certain size,
and then seeing how the number of zeros grows as the maximum size
considered is allowed to grow.  This is in complete analogy
to Nevanlinna's notion of a counting function to count the number of zeros
of a complex entire or meromorphic function.

For several variable polynomials, one generally does not try to ``count''
zeros.  Rather, one counts irreducible factors, usually weighted by the
degree of the irreducible factor.  For complex holomorphic functions
of several variables, including the case of complex polynomials, one
can define counting functions in a very geometric way by integrating
certain differential forms over the irreducible components
of the zero divisor of the function; see \textit{e.g.,}
\cite{NoguchiOchiai} or \cite{Shabat}.

One approach to defining non-Archimedean counting functions
in several variables is the approach initiated by H\`a Huy Kho\`ai
\cite{KhoaiSV} and used by Vu Hoai An and Doan Quang Manh, \textit{e.g.,}
\cite{AnSV}, \cite{AnManhCartan}, and \cite{AnManh}.  Although this approach
is, in principle, aesthetically pleasing because of its definition in terms
of the geometry of the Newton polytope associated to a several variable
power series, in practice, working with counting functions defined in this
way seems to be rather difficult, and seems not to produce
particularly aesthetic proofs. For instance, the difficulty in working
with this notion of counting function seems to have something to do with
An and Manh's need for some of their restrictive hypotheses in
\cite{AnManh}.  Moreover, working with this definition
seems to obscure connections to proofs of similar results for polynomials.
In \cite{CherryYe}, Cherry and Ye preferred not to give an
\textit{a priori} natural definition of counting function, but rather
first proved \cite[Lem.~2.3]{CherryYe}
that starting with a power series of several variables,
the counting functions of the one variable power series
obtained by restricting to a sufficiently generic line through the origin
do not depend on the generic line chosen and can be expressed in terms
of the power series coefficients.  The pay-back for doing this work first
before giving what may seem like an unnatural definition for
the counting function is that one can then in a relatively straightforward
manner connect Cherry and Ye's notion of counting function with
$|f|_r$ through a Poisson-Jensen type formula \cite[Th.~3.1]{CherryYe}.
Then, one can work with $|f|_r$ in a relatively straightforward manner
and in close analogy with how one would naturally work with a several
variable polynomial.

Suppose
$$
    f=\sum_\gamma a_\gamma z^\gamma
$$
is an entire function on $\mathbf{A}^m.$ As earlier in this section,
let $r>0$ and let
$\mathbf{r}=(r,\dots,r).$ Cherry and Ye define the \textbf{unintegrated
counting function} of zeros of $f$ by
$$
    n_f(0,r) = \sup\{|\gamma| : |a_\gamma|\mathbf{r}^\gamma=|f|_r\}.
$$
This is the number of zeros, counting multiplicity, that $f$ has
with $\max |z_j|\le r$ on a sufficiently generic line through the origin.
Also, define
$$
    n_f(0,0)=\lim_{r\to0} n_f(0,r)=\min\{|\gamma| : a_\gamma\ne0\}.
$$
As is typical in Nevanlinna theory, it is more convenient to work
with the \textbf{integrated counting function} of zeros
$$
    N_f(0,r) = n_f(0,0)\log r + \int_0^r\big(n_f(0,t)-n_f(0,0)\big)
    \frac{dt}{t}.
$$
Immediately from the definition we see that if $f$ is a non-constant
entire function, then for $r\ge1,$
\begin{equation}\label{Nlogr}
    \log r \le N_f(0,r)+O(1).
\end{equation}
Cherry and Ye's Poisson-Jensen-Green Formula \cite[Th.~3.1]{CherryYe}
then says that there exists a constant $C_f$ depending on $f$ but not $r$
such that
\begin{equation}\label{poisson}
    N_f(0,r)=\log|f|_r+C_f
\end{equation}
for all $r.$
These counting functions count zeros of $f$ with multiplicity.

The following proposition for counting functions corresponds to
Corollary~\ref{factorsmaller}.

\begin{proposition}\label{countingfactorsmaller}
Let $f=gh$ be entire functions.  Then,
\begin{enumerate}
\item[(i)] $n_f(0,r)=n_g(0,r)+n_h(0,r)$ for all $r\ge0,$
\item[(ii)] $N_f(0,r)=N_g(0,r)+N_h(0,r)$ for all $r\ge0,$ and
\item[(iii)] $N_f(0,r)\ge N_g(0,r)$ for all $r\ge1.$
\end{enumerate}
\end{proposition}

\begin{proof}
The equality in~(ii) follows immediately from~(i) and the definition
of the integrated counting functions.  The inequality in~(iii) follows
from~(ii) and the fact that $N_h(0,r)\ge0$ if $r\ge1.$  Thus, we need
to show~(i). To do so, let
$$
    f(z)=\sum_\alpha a_\alpha z^\alpha, \qquad
    g(z)=\sum_\beta b_\beta z^\beta, \qquad\text{and}\qquad
    h(z)=\sum_\gamma c_\gamma z^\gamma.
$$
We leave the case $r=0$ for the reader.  Let $r>0.$
Let $\beta_0$ and $\gamma_0$ be the largest multi-indices in the
graded lexicographical ordering such that
$$
    |b_{\beta_0}|r^{|\beta_0|}=|g|_r \qquad\text{and}\qquad
    |c_{\gamma_0}|r^{|\gamma_0|}=|h|_r
$$
respectively.  By definition, $n_g(0,r)=|\beta_0|$ and
$n_h(0,r)=|\gamma_0|.$ If $|\alpha|>|\beta_0|+|\gamma_0|,$ then
$$
    |a_\alpha|r^{|\alpha|}
    \le \max_{\beta+\gamma=\alpha}|b_\beta|r^{|\beta|}
    |c_\gamma|r^{|\gamma|} < |g|_r|h|_r=|f|_r,
$$
where the second inequality follows from the fact that if
$$
    |\beta|+|\gamma|=|\alpha|>|\beta_0|+|\gamma_0|,
$$
then we must have
$$
    |\beta|>|\beta_0|\qquad\text{or}\qquad
    |\gamma|>|\gamma_0|.
$$
On the other hand, if we consider $\alpha_0=\beta_0+\gamma_0,$ then
$$
    \alpha_0=\sum_{\beta+\gamma=\alpha_0}b_\beta c_\gamma.
$$
If $\beta\ne\beta_0$ (and so $\gamma\ne\gamma_0$), then either
$\beta$ comes after $\beta_0$ or $\gamma$ comes after $\gamma_0$
in the graded lexicographical ordering, which means
$$
    |b_\beta|r^{|\beta|}|c_\gamma|r^{|\gamma|} <
    |b_{\beta_0}|r^{|\beta_0|}|c_{\gamma_0}|r^{|\gamma_0|},
$$
and so
$$
    |a_{\alpha_0}|r^{|\alpha_0|}=
    |b_{\beta_0}|r^{|\beta_0|}|c_{\gamma_0}|r^{|\gamma_0|}
    =|g|_r|h|_r=|f|_r.
$$
Thus,
$$
    n_f(0,r)=|\alpha_0|=|\beta_0|+|\gamma_0|=n_g(0,r)+n_h(0,r).\qedhere
$$
\end{proof}

In \cite{CherryYe}, Cherry and Ye did not discuss truncated counting
functions, where zeros are counted without multiplicity or with
their multiplicities ``truncated'' to a certain level.
In complex Nevanlinna theory, since one has a natural
geometric definition for counting functions defined as integrals
over irreducible components of an analytic divisor, it is straightforward
to define truncated counting functions.  Since Cherry and Ye's
definition of counting functions is given in terms of power series
coefficients, it is clear that there will be no obvious definition of
truncated counting functions in terms of the power series coefficients.
Instead, we use the nice ring theoretic properties of $\mathcal{E}_m$
discussed above and
\textit{in characteristic zero}
simply define \textbf{truncated counting functions}
by
$$
    n^{(1)}_f(0,r)=n_{R(f)}(0,r) \qquad\textnormal{and}\qquad
    N^{(1)}_f(0,r)=N_{R(f)}(0,r),
$$
where $R(f)$ denotes the radical of $f$ as discussed above.
Note that although $R(f)$ is only defined up to a multiplicative
constant, $n^{(1)}$ and $N^{(1)}$ are well-defined.  In characteristic
zero, Proposition~\ref{radical} justifies calling the counting functions
of the radical the ``truncated'' counting function for $f$ because
each irreducible factor of $f$ appears with multiplicity one in $R(f).$
In positive characteristic, $N_{R(f)}$ might be called
``overly truncated''
because it completely ignores all irreducible
factors of $f$ which appear with multiplicity divisible by the
characteristic.  We will see in section~\ref{truncated} how to
define truncated counting functions in positive characteristic that
include all irreducible factors. However, as we saw in Theorem~\ref{basic},
sometimes in positive characteristic, we can give lower bounds on
these overly truncated counting functions.

We complete this section by pointing out that
Boutabaa and Escassut \cite{BoutabaaEscassut}
were the first to work out one variable
non-Archimedean Nevanlinna theory in
positive characteristic.  Their work also highlights that in working
with Nevanlinna theory in positive characteristic, one often may ignore
zeros whose multiplicity is divisible by the characteristic.

\section{Hasse Derivatives and Generalized Wronskians}\label{Hasse}
If $\mathbf{F}$ has characteristic zero, then a formal power
series $f$ in the variables
$$
    z=(z_1,\dots,z_m)
$$
is non-constant if and only if at least one of its formal partial derivatives
$\partial f/\partial z_j$ is not identically zero.  By contrast,
if $\mathbf{F}$ has positive characteristic $p,$ then any
formal power series in $z^p=(z_1^p,\dots,z_m^p)$ is such
that all its partial derivatives $\partialgamma f$ are identically
zero for all $|\gamma|>0.$  Also, if $\mathbf{F}$ has positive characteristic
$p,$ then if $\gamma$ is a multi-index such that $\gamma_i\ge p$ for
some $i,$ then $\partialgamma f=0$ for all $f.$  Therefore, we will
introduce a modification of the standard
derivative, known as the Hasse derivative, which is more useful in
positive characteristic.

If $\alpha=(\alpha_1,\dots,\alpha_m)$ and $\beta=(\beta_1,\dots,\beta_m)$
are multi-indices, we use $\alpha+\beta$ to denote the multi-index
$$
    \alpha+\beta=(\alpha_1+\beta_1,\dots,\alpha_m+\beta_m).
$$
We say that $\alpha\ge\beta$ if
$\alpha_i\ge\beta_i$ for all $i$ from $1$ to $m.$
Note that this notion of $\ge$ is not a total ordering on the set
of multi-indices and is \textit{not}
the graded lexicographical ordering that
was used in the proof of Proposition~\ref{countingfactorsmaller}.
If $\alpha\ge\beta,$ we use $\alpha-\beta$ to denote the
multi-index
$$
    \alpha-\beta=(\alpha_1-\beta_1,\dots,\alpha_m-\beta_m).
$$
Also, if $\alpha\ge\beta,$ define the multinomial coefficient
$\binom{\alpha}{\beta}$ by
$$
\binom{\alpha}{\beta}=\binom{\alpha_1}{\beta_1}\cdots\binom{\alpha_m}{\beta_m},
$$
where the $\binom{\alpha_i}{\beta_i}$ are the standard binomial coefficients.
Given a formal power series
$$
    f = \sum_\alpha a_\alpha z^\alpha
$$
and a multi-index $\gamma,$ we
define the \textbf{Hasse derivative} of multi-index $\gamma$ of $f,$
which we will denote $\Dgamma f,$ to be the formal power series
defined by
$$
    \Dgamma f = \sum_{\alpha\ge\gamma}\binom{\alpha}{\gamma}a_\alpha
    z^{\alpha-\gamma}.
$$
Note that if $\gamma=(0,\dots,0),$ then $\Dgamma f =f$
and that if $|\gamma|=1,$ then $\Dgamma f=\partialgamma f.$
Given a $j$ from $1$ to $m$ and a positive integer $k,$ we will
use $D_j^k$ as a short-hand notation for $\Dgamma f$ where
$\gamma=(\gamma_0,\dots,\gamma_m)$ with $\gamma_j=k$ and
$\gamma_i=0$ for $i\ne j.$

Because the multinomial coefficients $\binom{\alpha}{\gamma}$
are integers and hence have non-Archimedean
absolute value at most $1,$ we see
that if \hbox{$\mathbf{r}=(r_1,\dots,r_m)$} is an $m$-tuple of
non-negative real numbers such that
$$
    \lim_{|\alpha|\to\infty} |a_\alpha|\mathbf{r}^\alpha=0,
$$
then
$$
    \lim_{|\alpha|\to\infty} \left|\binom{\alpha}{\gamma}a_\alpha\right|
    \mathbf{r}^{\alpha-\gamma} \le
    \frac{1}{\mathbf{r}^\gamma}\lim_{|\alpha|\to\infty}
    |a_\alpha|\mathbf{r}^\alpha=0,
$$
and so we see that if $f$ is in $\mathcal{E}_m,$ then $\Dgamma f$ is also
in $\mathcal{E}_m.$

Clearly,
$$
    \partialgamma f = \gamma! \Dgamma f, \qquad
    \textnormal{where~}\gamma! = \gamma_1!\cdots\gamma_m!.
$$
Thus, in characteristic zero, the Hasse derivatives are just constant
multiples of the ordinary derivatives, and so one sees immediately that
they have similar properties to those of the ordinary partial derivative.
In positive characteristic, one must check these.

\begin{proposition}\label{Hasseprops}
The Hasse derivatives satisfy the following basic properties:
\begin{enumerate}
\item[(i)] $\displaystyle \Dalpha[f+g]=\Dalpha f + \Dalpha g;$
\item[(ii)] $\displaystyle \Dalpha[fg] = \sum_{\beta+\gamma=\alpha}\Dbeta f \Dgamma g;$
\item[(iii)] $\displaystyle \Dalpha \Dbeta f = \binom{\alpha+\beta}{\beta}D^{\alpha+\beta}\!f.$
\item[(iv)] If $\mathbf{F}$ has positive characteristic $p$ and $s\ge0$
is an integer, then
$$
    D_i^{p^s}f^{p^s}= (D_i f)^{p^s}.
$$
\end{enumerate}
\end{proposition}

\begin{proof}
Property~(i) is obvious.  To check property~(ii), write out both sides
and compare like powers of $z.$ What is needed for equality is that
for multi-indices $\delta$ and $\epsilon$ with $\delta\ge\beta$
and $\epsilon\ge\gamma,$ one has
$$
    \sum_{\beta+\gamma=\alpha}\binom{\delta}{\beta}\binom{\epsilon}{\gamma}
    =\binom{\delta+\epsilon}{\alpha},
$$
which is nothing other than Vandermonde's Identity.
To check property~(iii), one needs the elementary identity that
for $\gamma\ge\alpha+\beta,$
$$
    \binom{\alpha+\beta}{\beta}\binom{\gamma}{\alpha+\beta}
    =\binom{\gamma}{\alpha}\binom{\gamma-\alpha}{\beta}.
$$
What one needs for~(iv) is that fact that for any integer $j,$
$$
    \binom{jp^s}{p^s}\equiv j \textrm{~mod~}p,
$$
which follows immediately from Lucas's Theorem.
\end{proof}

We also want to point out that as with ordinary partial derivatives,
the same proof as in Lemma~\ref{LDL} gives

\begin{lemma}[Logarithmic Derivative Lemma]\label{HasseLDL}
Let $f$ be an entire function in $\mathcal{E}_m$ and let $\gamma$
be a multi-index. Then,
$$
    |\Dgamma f|_r \le \frac{|f|_r}{r^{|\gamma|}}.
$$
\end{lemma}

\begin{corollary}\label{fnotdivideDf}
Let $f$ be an entire function in $\mathcal{E}_m$ and let
$\gamma$ be a multi-index with $|\gamma|>0.$ If $f$ divides
$\Dgamma f,$ then $\Dgamma f\equiv0.$
\end{corollary}

\begin{proof}
This follows from Lemma~\ref{HasseLDL} and Corollary~\ref{factorsmaller}.
\end{proof}

We will denote the fraction field of $\mathcal{E}_m$ by $\mathcal{M}_m$
and call it the field of \textbf{meromorphic functions} on
$\mathbf{A}^m.$ One sees immediately that one can use
Proposition~\ref{Hasseprops}~(ii) to inductively
extend the Hasse derivatives
to the field $\mathcal{M}_m$ and that the four properties
of Proposition~\ref{Hasseprops} continue to hold for functions
in $\mathcal{M}_m.$

For each integer $k\ge2,$ let
$$
    \mathcal{M}_m[k] = \{Q \in \mathcal{M}_m : D_j^i Q \equiv 0
    \text{~for~all~}0<i < k \text{~and~}1\le j \le m\}.
$$
If $\mathbf{F}$ has positive characteristic $p$ and if $s$ is a
positive integer, let
$$
    \mathcal{E}_m[p^s] = \{g^{p^s} : g \in \mathcal{E}_m\}.
$$
Note that $\mathcal{E}_m[p^s]$ is a subring of $\mathcal{E}_m$
and that it consists of those elements $f$ in $\mathcal{E}_m$
that can be written as convergent power series
in
$$
    z^{p^s}=(z_1^{p^s},\dots,z_m^{p^s}).
$$

\begin{proposition}\label{pthpower}
We have the following, depending on the characteristic of $\mathbf{F}$:
\begin{enumerate}
\item[(A)]If $\mathbf{F}$ has characteristic $0,$ then for all $k\ge2,$
we have $\mathcal{M}_m[k]=\mathbf{F}.$
\item[(B)] If $\mathbf{F}$ has positive characteristic $p$
and if $s$ is an integer $\ge1,$
then
\begin{enumerate}
\item[(B1)]$\mathcal{M}_m[p^{s-1}+1]=\mathcal{M}_m[p^s]$
\item[(B2)]$\mathcal{M}_m[p^s]$ is the fraction field
of $\mathcal{E}_m[p^s]$ and $D_i^{p^s}$ for $i=1,\dots,m$
are derivations on $\mathcal{M}_m[p^s].$
\end{enumerate}
\end{enumerate}
\end{proposition}

\begin{proof}
Clearly, we have~(A). Proposition~\ref{Hasseprops}~(iii)
implies~(B1).
We show~(B2) by induction on $s.$
Let $Q$ be an element of  $\mathcal{M}_m[2]=\mathcal{M}_m[p].$
Then $\partial Q/\partial z_i\equiv0$ for all $i=1,\dots,m.$
Write $Q=f/g$ with $f$ and $g$ relatively prime in $\mathcal{E}_m.$
Suppose that $f$ is not in $\mathcal{E}_m[p].$  Then, there is
an irreducible element $P$ of $\mathcal{E}_m$ such that $P$ divides $f$
to a multiplicity not divisible by $p$ and such that
$\partial P/\partial z_i\not\equiv0$ for some $i.$
This implies that $P$ divides $f$ to a higher power than it divides
$\partial f/\partial z_i.$  Because $\partial Q/\partial z_i\equiv0,$
this would then imply $P$ must divide $g,$ contradicting the fact
that $f$ and $g$ are relatively prime.  Hence $f$ must have been
in $\mathcal{E}_m[p].$ Similarly, $g$ must be in $\mathcal{E}_m[p].$
That $D_i^p$ is a derivation
on $\mathcal{M}_m[p]$ then follows from
Proposition~\ref{Hasseprops}~(ii) or~(iv).
By~(B1), we have that
$$
    \mathcal{M}_m[p^{s+1}]=\{Q \in \mathcal{M}_m[p^s] :
    D_i^{p^s}Q=0 \textnormal{~for~all~}i=1,\dots,m\},
$$
and so the proof is completed by induction. Indeed, writing
an element $Q$ of \hbox{$\mathcal{M}_m[p^{s+1}]\subset\mathcal{M}_m[p^s]$}
as $f/g$ with $f$
and $g$ relatively prime elements of $\mathcal{E}_m[p^s]$
arguing as before using $D_i^{p^s}$ in place of the first partials,
we see that every irreducible element $P$ of $\mathcal{E}_m$ that divides
either $f$ or $g$ must divide them with multiplicity divisible
by $p^{s+1},$ and hence $f$ and $g$ must be in $\mathcal{E}_m[p^{s+1}].$
\end{proof}

\begin{theorem}[{Hsia-Wang \cite[Lem.~2]{HsiaWang}}]\label{gwronsk}
Let $\mathbf{F}$ have characteristic zero (resp.\ positive
characteristic $p$)
and let $s\ge1$
be an integer.
Let $f=(f_0,\dots,f_{n-1})$ be an $n$-tuple of
entire functions. For a multi-index $\gamma,$
let
$$
    \Dgamma f = (\Dgamma f_0,\dots,\Dgamma f_{n-1}).
$$
Let $\gamma^0$ be the multi-index $(0,\dots,0).$
If $f_0,\dots,f_{n-1}$ are linearly independent over
$\mathbf{F}$ (resp.\ $\mathcal{M}_m[p^s],$) then there exist
multi-indices $\gamma^1,\dots,\gamma^{n-1}$ such that
$$
    |\gamma^i|\le |\gamma^{i-1}|+1 \qquad
    \text{(resp.~} |\gamma^i|\le |\gamma^{i-1}|+p^{s-1}\text{)}
$$
and such that
$$
    \det \left ( \begin{array}{ccc} f_0 \;&\; \dots \;&\; f_{n-1} \\
    D^{\gamma^1}\!f_0 \;&\; \dots \;&\;D^{\gamma^1}\!f_{n-1} \\
    D^{\gamma^2}\!f_0 \;&\; \dots \;&\;D^{\gamma^2}\!f_{n-1} \\
    \vdots\;&\;\vdots\;&\;\vdots \\
    D^{\gamma^{n-1}}\!f_0 \;&\; \dots \;&\;
    D^{\gamma^{n-1}}\!f_{n-1}
    \end{array} \right ) \not\equiv 0.
$$
\end{theorem}

\begin{remark*} The determinant in Theorem~\ref{gwronsk} is called a
\textbf{generalized Wronskian}. For polynomials in characteristic zero,
this theorem, with a different proof, appears in
\cite{Roth}.
In the case of complex entire functions of several variables,
a similar theorem can be found in \cite{Fujimoto}.
\end{remark*}

\begin{remark*}Often, \textit{e.g.,} \cite{HsiaWang},
one tends to state this lemma with
$|\gamma^i|\le i$ (resp.\ $|\gamma^i|\le ip^{s-1}$), but we
will want to give a lower bound on $\sum |\gamma^i|$ in terms
of $|\gamma^{n-1}|,$ and so for us the observation that
$|\gamma^{i-1}|\ge|\gamma^i|-1$ (resp.\
$|\gamma^{i-1}|\ge|\gamma^i|-p^{s-1}$) is important.
\end{remark*}

\begin{remark*} We also remark here that Theorem~\ref{gwronsk}
can be used to derive a positive characteristic
Cartan-type second main theorem for linearly non-degenerate
non-Archimedean analytic curves in projective space.
For instance, the proof given in \cite[Th.~5.1]{CherryYe}
goes through once a non-vanishing generalized Wronskian exists.
\end{remark*}

\begin{proof}
We write the proof in the case of positive characteristic.
The same proof works in characteristic zero by
Proposition~\ref{pthpower}~(A) if all powers of $p$ are replaced by $1.$

We proceed by induction on $n.$ When $n=1,$ the theorem is trivial.
Now assume that the theorem is true for $n-1.$
By the induction
hypothesis, there exist multi-indices
$\gamma^0,\gamma^1,\dots,\gamma^{n-2}$ with
$|\gamma^i|\le |\gamma^{i-1}|+p^{s-1}$
and such that the $D^{\gamma^i}\! f$ for \hbox{$i=0,\dots,n-2$}
span an $n-1$ dimensional $\mathcal{M}_m$ vector subspace
of $\mathcal{M}_m^n.$ Let $k=|\gamma^{n-2}|+p^{s-1}.$
Let $V$ be the
$\mathcal{M}_m$ vector subspace of  $\mathcal{M}_m^n$
spanned by $\Dgamma f$
for all $|\gamma|\le k.$
 If the theorem were not
true, then $V$ could not have dimension $n,$ and so there exist
\hbox{$Q_0,\dots,Q_{n-1}$} not all zero in $\mathcal{M}_m$ such that
\begin{equation}\label{allvanish}
    Q_0 \Dgamma f_0 + \dots + Q_{n-1} \Dgamma f_{n-1} \equiv 0
\end{equation}
for every $\gamma$ with $|\gamma|\le k.$
Because the vectors
$$
    (D^{\gamma^0}\!f_j,\dots,D^{\gamma^{n-2}}\!f_j)
    \quad\text{for}\quad j=0,\dots,n-2
$$
are linearly independent over $\mathcal{M}_m$
by the induction hypothesis, we can assume $Q_{n-1}\equiv1,$ and hence we
have
$$
    Q_0 \Dgamma f_0 + \dots + Q_{n-2} \Dgamma f_{n-2} + f_{n-1} \equiv 0
    \qquad\textnormal{for all~}|\gamma|\le k.
$$

Our goal is to show that the $Q_j$ are in $\mathcal{M}_m[p^s].$
Note that for \hbox{$i=0,\dots,n-2$} and \hbox{$\ell=1,\dots,m,$}
\begin{align*}
    0&=D_\ell^1 \left[\sum_{j=0}^{n-2} Q_j D^{\gamma^i}\!f_j
        + D^{\gamma^i}f_{n-1}\right]\\
    &=\sum_{j=0}^{n-2}D_\ell^1 Q_j D^{\gamma^i}\!f_j+ \sum_{j=0}^{n-1}Q_jD_\ell^1
    D^{\gamma^i}\!f_j\\
    &=\sum_{j=0}^{n-2}D_\ell^1 Q_j D^{\gamma^i}\!f_j,
\end{align*}
where the last line follows from Proposition~\ref{Hasseprops}~(iii)
and equation~(\ref{allvanish}).  By the linear independence
of the vectors $(D^{\gamma^0}\!f_j,\dots,D^{\gamma^{n-2}}\!f_j),$
we conclude that $D_\ell^1 Q_j\equiv0$ for all $\ell$ and all
\hbox{$j=1,\dots,n-2.$}
Thus, the $Q_j$ belong to \hbox{$\mathcal{M}_m[2]=\mathcal{M}_m[p].$}
Now assume the $Q_j$ belong to $\mathcal{M}_m[p^t]$
for some $t\ge1.$ By Proposition~\ref{pthpower}~(B),
we can apply $D_\ell^{p^{t+1}}$ as if it were a derivation to get
\begin{align}
    0&=D_\ell^{p^{t+1}}\!\! \left[\sum_{j=0}^{n-2} Q_j D^{\gamma^i}\!f_j
        + D^{\gamma^i}\!f_{n-1}\right]\nonumber \\
    &=\sum_{j=0}^{n-2}D_\ell^{p^{t+1}}\!\! Q_j D^{\gamma^i}\!f_j+ \sum_{j=0}^{n-1}Q_jD_\ell^{p^{t+1}}\!\!
    D^{\gamma^i}\!f_j. \label{rightvanish}
\end{align}
If $t<s,$ we can use Proposition~\ref{Hasseprops}~(iii),
the fact that
$$
    |\gamma^i|\le |\gamma^{n-2}| \le (k-1)p^{s-1}
$$
and equation~(\ref{allvanish}) to conclude that the right-hand sum
in~(\ref{rightvanish}) vanishes, and thus,
$$
    0=\sum_{j=0}^{n-2}D_\ell^{p^{t+1}}\!\! Q_j D^{\gamma^i}\!f_j.
$$
Again, by linear independence, we conclude $D_\ell^{p^{t+1}}\!\! Q_j\equiv0$
and so the $Q_j$ belong to $\mathcal{M}_m[p^{t+1}].$ Continuing in this manner,
we find that the $Q_j$ are in $\mathcal{M}_m[p^s],$ contradicting
the assumption that the $f_j$ are linearly independent over
$\mathcal{M}_m[p^s].$
\end{proof}

\begin{proposition}\label{divDgamma}
Let $f$ be an entire function in $\mathcal{E}_m.$
Let $\gamma=(\gamma_1,\dots,\gamma_m)$ be a multi-index.
Let $P$ be an irreducible element of $\mathcal{E}_m$ that
divides $f$ with exact multiplicity $e.$
If $e>|\gamma|,$ then $P^{e-|\gamma|}$ divides
$\Dgamma f.$ Moreover, if \hbox{$\mathrm{char~}\mathbf{F}=p>0$}
and if $e$ is divisible by \hbox{$p^s>\max\{\gamma_1,\dots,\gamma_m\},$}
then $P^e$ divides $\Dgamma f.$
\end{proposition}

\begin{proof}
Because the $D_i$ commute, it suffices to show the proposition
for $\Dgamma=D_i^k.$

In the special case that \hbox{$\mathrm{char~}\mathbf{F}=p>0$}
and $e$ is divisible by \hbox{$p^s>\max\{\gamma_1,\dots,\gamma_m\},$}
that $P^e$ divides $D_i^k f$ follows easily from
Proposition~\ref{Hasseprops}~(ii) since
$$
    D_i^jP^{\ell p^s}=0 \qquad\text{for all~} 0<j<p^s
$$
and $p^s>k$ by assumption.

We show the general case  by induction
on $e$ and $k.$ The case $e=k$ is trivial.
We now suppose that the proposition holds for all
$D_i^j$ with $j\le k$ and for $e$ and then show that it also holds
for $k$ and $e+1.$
Suppose $f=P^{e+1}g$ with $g$ relatively
prime to $P.$ By Proposition~\ref{Hasseprops}~(ii),
$$
    D_i^k f = D_i^k(P\cdot P^eg)
    =PD_i^k(P^eg)+\sum_{j=1}^k D_i^j P D_i^{k-j}(P^eg).
$$
If $e>k\ge k-j,$ then by induction, $P^{e-(k-j)}$ divides
$D_i^{k-j}(P^eg)$ and hence $P^{e-k}$ divides
$D_i^{k-j}(P^eg)$ for all $j=0,\ldots,k$ and $P^{e+1-k}$
divides $D_i^{k-j}(P^eg)$ for all \hbox{$j>0.$}
\end{proof}

\section{Higher Radicals and Truncated Counting Functions}\label{truncated}
If $\mathbf{F}$ has characteristic zero, if $f$ is in $\mathcal{E}_m,$
and if $\ell\ge1$ is an integer, then clearly if $P$ is an irreducible
element of $\mathcal{E}_m$ that divides $f$ with multiplicity $e,$
then $P$ divides $\gcd(f,R(f)^\ell)$ with multiplicity
$\min\{\ell,e\}.$  Thus, \textit{in characteristic zero},
we can define the
\textbf{$\ell$-th truncated counting function} by
$$
    N^{(\ell)}_f(0,r)=N_{\gcd(f,R(f)^\ell)}(0,r).
$$

We saw at the end of section~\ref{prelim} that in positive
characteristic $p,$ the radical $R(f)$ does not contain those irreducible
factors of $f$ that divide $f$ with multiplicity divisible by $p.$
Although that was exactly what was appropriate in Theorem~\ref{basic},
when we consider
$$
    f_n=f_0+\dots+f_{n-1}
$$
with $n>2,$ we will not be able to ignore all
irreducible factors with multiplicity divisible by $p.$
Thus, we want to define
truncated counting functions in positive characteristic that
include all irreducible factors.

\textit{For the rest of this section,
let $\mathbf{F}$ have positive characteristic $p.$}
We will use the following proposition to inductively define
\textbf{higher $p^s$-radicals} for integers $s\ge1.$

\begin{proposition}\label{pradical}
Let $f$ be an entire function in $\mathcal{E}_m$ and let $s\ge1$ be
an integer.  Assume that we have defined a $p^{s-1}$-radical
$R_{p^{s-1}}(f)$ that has the property that $R_{p^{s-1}}(f)$ is
square-free and has the property that an irreducible element of
$\mathcal{E}_m$ divides $R_{p^{s-1}}(f)$ if and only if $P$ divides
$f$ with multiplicity not divisible by $p^{s}.$ Let
$$
    \bar f = \frac{f}{\gcd(f,R_{p^{s-1}}(f)^{p^s})}.
$$
For $i=1,\ldots,m,$ let $g_i=\gcd(\bar f,D_i^{p^s}\bar f),$
let $h_i=\bar f/g_i,$ and let $H$ be the least common multiple of the $h_i.$
Let
$$
    G=\frac{H}{\gcd(H,R_{p^{s-1}}(H)^{p^s-1})}.
$$
Then,
\begin{enumerate}
\item[(i)] If $P$ is an irreducible element of $\mathcal{E}_m$
that divides $G,$ it divides $G$ with multiplicity exactly $p^s.$
\item[(ii)] If $P$ is an irreducible element of $\mathcal{E}_m,$
then $P$ divides $G$ if and only if $P$ divides $f$ with multiplicity
a multiple of $p^s$ but not a multiple of $p^{s+1}.$
\end{enumerate}
It follows that $G$ is a $p^s$-th power, so let $R$ be a
$p^s$-th root of $G$ and let $R_{p^s}(f)$ be the least common
multiple of $R_{p^{s-1}}$ and $R.$
\end{proposition}

\begin{proof}
Our induction begins with the radical as defined in
section~\ref{prelim}, so we let $R_{p^0}(f)=R(f).$
To show the inductive step,
let $P$ be an irreducible element of $\mathcal{E}_m$ that divides
$\bar f.$ Note that if $P$ divides $\bar f,$ then it divides it with
multiplicity at least $p^s.$ Write
$$
    \bar f=P^{jp^s+e}\tilde f,
$$
where $j\ge1$ and $0\le e < p^s$ are integers and $\tilde f$ is relatively
prime to $P.$
Then, by Proposition~\ref{Hasseprops}~(ii) and~(iv)
and Proposition~\ref{pthpower}~(B), for $i=1,\ldots,m,$
\begin{equation}\label{pradicaleqn}
    D_i^{p^s}\bar f = D_i^{p^s}[P^{jp^s}P^e\tilde f]=
    P^e\tilde f (D_i[P^j])^{p^s}+P^{jp^s}D_i^{p^s}[P^e\tilde f].
\end{equation}

We first consider the case that $j$ is not divisible by $p.$
Because $P$ is irreducible and hence not a $p$-th power,
there exists an $i$ such that $D_iP\not\equiv0.$
Because
$$
    D_i(P^j)=jP^{j-1}D_iP,
$$
we see from the assumption that
$j$ is not divisible by $p,$
that $P$ divides
$D_i^{p^s}f$ with exact multiplicity $(j-1)p^s+e,$ and so by
Proposition~\ref{divDgamma}, $P$ divides $H$ with exact multiplicity
$p^s.$

In case that $j$ is divisible by $p,$ we see from
equation~(\ref{pradicaleqn}) that $P$ divides $H$ with multiplicity
at most $e<p^s.$ Thus, $P$ does not divide $R.$
\end{proof}

We now show the existence of the square free part of an entire
function, which is square free and contains all the irreducible factors
dividing the function.

\begin{theorem}\label{sqfree}
Let $f$ be an entire function in $\mathcal{E}_m.$ There exists an
entire function $S(f)$ in $\mathcal{E}_m$ such that $S(f)$ is square
free and such that an irreducible element $P$ in $\mathcal{E}_m$
divides $S(f)$ if and only if it divides $f.$
\end{theorem}

We will call $S(f)$ the \textbf{square free part} of $f.$
We define the \textbf{$\ell$-th truncated counting function} by
$$
    N^{(\ell)}_f(0,r) = N_{\gcd(f,S(f)^\ell)}(0,r).
$$
As in characteristic zero, the
$\ell$-th truncated counting function truncates all multiplicities
higher than $\ell$ to $\ell.$

\begin{proof}[Proof of Theorem~\ref{sqfree}.]
The proof is similar to the proof of the existence of greatest
common divisors given in \cite{CherryGCD}. The case that $f$ is
identically zero is trivial, so assume that $f$ is not identically
zero. For \hbox{$i=1,2,\dots,$} let $r_i$ be an increasing sequence
of elements of $|\mathbf{F}^\times|$ such that $r_i\to\infty.$
Consider $f$ as an element of the ring $\mathcal{A}_m(r_i)$ of
analytic functions on $\mathbf{B}^m(r_i),$ the closed ball of radius
$r_i.$ Let $z_0$ be a point in $\mathbf{B}^m(r_1)$ such that
$f(z_0)\ne0.$ Let $R_{p^s}(f)$ be the higher radicals of $f$ defined
as in Proposition~\ref{pradical} normalized so that for each $s,$ we
have $R_{p^s}(f)(z_0)=1.$ Because the ring $\mathcal{A}_m(r_i)$ is
factorial \cite[\S5.2.6, Th.~1]{BGR}, only finitely many of the
irreducible elements in $\mathcal{E}_m$ that divide $f$ are
non-units in $\mathcal{A}_m(r_i).$ Each of these divides $f$ to some
finite multiplicity, and so there exists some $s_i$ such that every
irreducible element in $\mathcal{E}_m$ that divides $f$ and is not a
unit in $\mathcal{A}_m(r_i)$ also divides $R_{p^s}(f)$ for all $s\ge
s_i.$ This means that for $s,t\ge s_i,$ $R_{p^s}(f)$ and
$R_{p^t}(f)$ differ by a unit in $\mathcal{A}_m(r_n).$ Let
$u_{i,i+1}$ be the unit in $\mathcal{A}^m(r_i)$ such that
$$
    R_{p^{s_i}}(f) = u_{i,i+1}R_{p^{s_{i+1}}}(f),
$$
and note that $u_{i,i+1}(z_0)=1.$  Then, writing $u_{i,i+1}$
as a power series about $z_0,$ we have
$$
    u_{i,i+1}(z)=1+\sum_{|\gamma|\ge1}a_\gamma(z-z_0)^\gamma
    \qquad\text{with~}|a_\gamma|r_i^{|\gamma|} < 1
    \text{~for all~}|\gamma|\ge1.
$$
Thus, for $j>i,$
$$
    |u_{j,j+1}-1|_{r_{i}} < \frac{r_i}{r_j}.
$$
Since $r_i/r_j\to0$ as $j\to\infty$ with $i$ fixed, we can define
units $v_i$ in $\mathcal{A}_m(r_i)$ by
$$
    v_i=\prod_{j=i}^\infty u_{j,j+1}.
$$
For $j>i,$ we have
\begin{align*}
    R_{p^{s_j}}(f)v_i&=R_{p^{s_j}}(f)\prod_{k=i}^\infty u_{k,k+1}\\
    &=R_{p^{s_j}}(f)\left(\prod_{k=i}^{j-1}u_{k,k+1}\right)
    \left(\prod_{k=j}^\infty u_{k,k+1}\right)=R_{p^{s_i}}(f)v_j.
\end{align*}
This precisely means that $R_{p^{s^i}}(f)v_i^{-1}$ converges to an entire
function $F$ as $i\to\infty$ because the difference between
$R_{p^{s^i}}(f)v_i^{-1}$ and $R_{p^{s^j}}(f)v_j^{-1}$ is identically zero
on $\mathbf{B}^m(r_i).$ Note also that $v_iF=R_{p^{s_i}}(f)$
in $\mathcal{A}_m(r_i).$
\par
We claim that $F$ is square free
and that an irreducible element $P$ of $\mathcal{E}_m$ divides $F$ if and
only if it divides $f.$
\par
To show that $F$ is square free, suppose that $P$ is an irreducible
element of $\mathcal{E}_m$ such that $P^2$ divides $F.$
This means $P^2$ divides $R_{p^{s_i}}(f)$ in $\mathcal{A}_m(r_i).$
For $i$ sufficiently large, $P$ is not a unit in $\mathcal{A}_m(r_i),$
and so we would have that $R_{p^{s_i}}(f)$ is not square free
in $\mathcal{A}_m(r_i).$  However, the proof of
Proposition~\ref{pradical} works equally well in the ring
$\mathcal{A}_m(r_i),$ and thus $R_{p^{s_i}}(f)$ is also square free
in $\mathcal{A}_m(r_i).$
\par
Finally, let $P$ be an irreducible element of $\mathcal{E}_m.$
Suppose $P$ divides $f.$  Then, $P$ divides $R_{p^{s_i}}(f)$
for all $i$ sufficiently large.  In other words, there exist
analytic functions $h_i$ in $\mathcal{A}_m(r_i)$ such that
$$
    Ph_i=R_{p^{s_i}}(f).
$$
Becase $Ph_iv_i^{-1}$ converges to $F$ as $i\to\infty,$
for $j>i,$ we have
$$
    P(h_iv_i^{-1}-h_jv_j^{-1})=0 \qquad\text{in~}\mathcal{A}_m(r_i).
$$
Thus
$h_iv_i^{-1}$ converges to an entire function $H$ such that $PH=F,$
and so $P$ divides $F.$  For the other direction, suppose that $P$
divides $F.$  Then, $P$ divides $R_{p^{s_i}}(f)$ in $\mathcal{A}_m(r_i),$
and so again noticing that the proof of Proposition~\ref{pradical}
also works for $\mathcal{A}_m(r_i),$ we get that $P$ divides $f$
in $\mathcal{A}_m(r_i).$  In other words, there exist analytic functions
$g_i$ in $\mathcal{A}_m(r_i)$ such that $Pg_i=f.$  This implies
that the $g_i$ converge to an analytic function $G$ such that
$PG=f,$ and hence $P$ divides $f$ in $\mathcal{E}_m.$
\end{proof}

\section{Linear Algebra}\label{linalg}

Let $V$ be a vector space over a field $\mathbf{E}.$
Let $v_0,\ldots,v_n$ be $n$ linearly dependent vectors in
$V.$  Call an index set $I\subset\{0,\dots,n\}$
\textbf{minimal} if the set of vectors
$$
    \{v_i : i \in I\}
$$
are linearly dependent, but such that for every proper
subset $I'\subsetneq I,$ the sets of vectors
$\{v_i : i \in I'\}$ are linearly independent.

\begin{lemma}[Brownawell-Masser]\label{bm}
Let $v_0,\ldots,v_n$ be $n+1$ vectors in a vector space $V$ over
a field $\mathbf{E}$ such that $\sum v_i=0.$  Assume that no proper
subsum vanishes, \textit{i.e.,}
$$
    \sum_{i\in I} v_i\ne0 \text{~for all proper~}
    I\subsetneq\{1,\dots,n\}.
$$
Then, there exists an integer
$u\ge1,$ a partition
$$
    \{0,\dots,n\}=I_0\cup\ldots\cup I_{u-1},
$$
and non-empty subsets
$$
    J_\ell \subset \bigcup_{j=1}^\ell I_j, \qquad\ell=0,\ldots,u-2
$$
such that $I_0$ and $I_j\cup J_{j-1}$ for $j=1,\ldots,u-1$ are minimal.
\end{lemma}

\begin{proof}
If $\{0,\ldots,n\}$ is minimal, set $I_0=\{0,\ldots,n\}.$
If $\{0,\ldots,n\}$ is not minimal, see \cite[Lem.~6]{BrownawellMasser}.
\end{proof}

In positive characteristic $p,$ we want to apply Theorem~\ref{gwronsk}
to entire functions linearly independent over $\mathbf{F},$ so we
complete this section by proving that if a collection of functions
are linearly independent over $\mathbf{F},$ then they are also
linearly independent over $\mathcal{M}_m[p^s]$ for some integer
$s\ge1.$

\begin{lemma}\label{linindep}
Let $f_1,\ldots,f_n$ be meromorphic functions in $\mathcal{M}_m$
linearly independent over $\mathbf{F},$
with \hbox{$\mathrm{char~}\mathbf{F}=p>0.$}
Then there exists an
integer $s\ge1$ such that $f_1,\ldots,f_n$ are linearly independent
over $\mathcal{M}_m[p^s].$
\end{lemma}

\begin{proof}
Suppose the lemma is not true.  Then, $f_1,\ldots,f_n$
are linearly dependent over $\mathcal{M}_m[p^s]$ for every $s\ge1.$
For each $s\ge1,$ let $I_s\subset\{1,\ldots,n\}$ be minimal.
Note that each $I_s$ contains at least two indices, otherwise
one of the functions $f_j$ would be identically zero, and hence
the $f_j$ could not be linearly independent over $\mathbf{F}.$
Because there are only finitely many possible subsets $I_s,$
we may assume without loss of generality that
$I_s=\{1,\ldots,t\}$ for infinitely many $s.$
Thus, for infinitely many $s,$ we have that $f_2,\ldots,f_t$ are linearly
independent over $\mathcal{M}_m[p^s]$ and that $f_1$ is in the
$\mathcal{M}_m[p^s]$ linear span of $f_2,\ldots,f_t.$  In other
words, there exist unique $Q_{s,2},\ldots,Q_{s,t}$ in $\mathcal{M}_m[p^s]$
for infinitely many $s$ such that
$$
    f_1=Q_{s,2}f_2+\ldots+Q_{s,t}f_t.
$$
On the other hand, $\mathcal{M}_m[p^{s'}]\subset\mathcal{M}_m[p^s]$
if $s'\ge s,$ so by the linear independence of $f_2,\ldots,f_t,$
the $Q_{s,j}$ do not depend on $s.$  Hence, $Q_{s,j}$ is in
$\mathcal{M}_m[p^s]$ for infinitely many $s,$ and must therefore
be in $\mathbf{F}.$  This contradicts the linear independence of the
$f_j$ over $\mathbf{F}.$
\end{proof}

If $\mathbf{F}$ has positive characteristic $p,$ then we define the
\textbf{index of independence} of a collection $\mathcal{F}$
of entire or meromorphic functions to be the smallest
integer $s$ such that any subset of functions in $\mathcal{F}$
linearly independent over $\mathbf{F}$ remains linearly independent
over $\mathcal{M}_m[p^s],$ provided such an integer exists.
Lemma~\ref{linindep} shows that such an integer always exists
if $\mathcal{F}$ is finite.

\section{ABC Theorems}\label{ABC}

In the three function ABC theorem, one begins with
$f_2=f_0+f_1$ with the $f_i$ relatively prime.
Note that here $\gcd(f_0,f_1,f_2)=1$ implies that the $f_j$
are also \textit{pairwise} relatively prime because by the linear
dependence if two of the functions have a common factor, it divides
the third as well. To generalize to $n+1$ functions, one obviously
wants to consider
$$
    f_n=f_0+\dots+f_{n-1},
$$
or more symmetrically,
$$
    0=f_0+\dots+f_{n}.
$$
Some work on such generalizations, \textit{e.g.,}
\cite{ShapiroSparer}, assumes the rather strong hypothesis
that the $f_j$ are pairwise relatively prime.  Other work,
\textit{e.g.,} \cite{BrownawellMasser}, assumes
$\gcd(f_0,\dots,f_n)=1$ and that for every proper
sub-index set $I\subsetneq \{0,\dots,n\},$
$$
    \sum_{i\in I}f_i \ne 0.
$$
This hypothesis is referred to as ``no vanishing subsums.''
The recent work of De~Bondt generalizes these two hypotheses to
the following: Let \hbox{$0=f_0+\dots+f_n$} and assume that for
each index set $I\subseteq \{0,\dots,n\},$ if
$$
    \sum_{i\in I}f_i =0, \quad\textnormal{then}\quad
    \gcd(\{f_i : i\in I\})=1.
$$
\par
In the three function ABC theorem, the right-hand side of the
ABC inequality involves the radical $R(f_0f_1f_2)$ of the product.
But because the functions are pairwise relatively prime, this
is the same as the product of the radicals:
$R(f_0)R(f_1)R(f_2).$  When one begins with $n+1$ functions
that are not necessarily pairwise relatively prime,
then the square free part of the product, $S(f_0\cdots f_n)$
will not in general be the same as the product of the
square free parts: $S(f_0)\cdots S(f_n).$ Again, we follow
De~Bondt's lead by presenting generalized ABC inequalities
of both types.

The following two theorems are our generalized ABC theorems
for non-Archimedean entire functions of several variables.

\begin{theorem}[Generalized ABC Theorem (First Version)]\label{abcsum}
Let $f_0,\ldots,f_n$ be \hbox{$n\ge2$}
entire functions in $\mathcal{E}_m,$
not all of which are constant and none of which are identically zero.
Assume
\begin{equation}\label{sumzero}
    0=f_0+\ldots+f_n
\end{equation}
and assume that for each index set $I\subseteq\{0,\dots,n\},$
\begin{equation}\label{subsum}
    \text{if}\quad \sum_{i\in I} f_i=0 \quad\text{then}\quad
    \gcd(\{f_i : i \in I\})=1.
\end{equation}
Let $2\le d\le n$ be the dimension of the $\mathbf{F}$ vector space
spanned by the $f_i.$
If $\mathbf{F}$ has characteristic zero, let \hbox{$c=1,$} and if
$\mathbf{F}$ has positive characteristic $p,$ let $c=p^{s-1},$
where $s$ is the index of independence for the $f_i.$
Then, there exist integers $a$ and $b$ with
$$
    1\le a \le c(d-1) \qquad\text{and}\qquad
    b\ge a\left\lceil\frac{a}{c}\right\rceil-
    \frac{\displaystyle\left\lceil\frac{a}{c}\right\rceil
    \left(\left\lceil\frac{a}{c}\right\rceil-1\right)}{2}c\ge a
$$
such that for $r\ge1,$
\begin{equation}\label{abctrsum}
    \max_{0\le j \le n}\log|f_j|_r \le
    \sum_{j=0}^n N^{(a)}_{f_j}(0,r) - b\log r + O(1).
\end{equation}
Moreover, if $\mathrm{char~}\mathbf{F}=p>0$ there further
exists a non-negative integer $\sigma$
with \hbox{$p^\sigma\le a$}
such that for $r\ge1,$
\begin{equation}\label{abctrsumcharp}
    \max_{0\le j \le n}\log|f_j|_r \le
    \sum_{j=0}^n N_{G_j}(0,r)-b\log r + O(1),
\end{equation}
where $G_j=\gcd(f_j,R_{p^\sigma}(f_j)^a).$
\end{theorem}

\begin{theorem}[Generalized ABC Theorem (Second Version)]\label{abcprod}
Let $f_0,\ldots,f_n$ be $n\ge2$ entire functions in $\mathcal{E}_m,$
not all of which are constant and none of which are identically zero.
Assume
\begin{equation*}
    0=f_0+\ldots+f_n.
\end{equation*}
Let $2\le d\le n$ be the dimension of the $\mathbf{F}$ vector space
spanned by the $f_i.$ Let $2\le k \le n$ and assume that for every
set of $k$ distinct indices
\begin{equation}\label{krelprime}
    0 \le i_1 < i_2 < \dots < i_k \le n,
    \quad\text{we have}\quad
    \gcd(f_{i_1},\dots,f_{i_k})=1.
\end{equation}
Let $\bar k=\min\{k,d\}.$
If $\bar k>2,$ further assume

\begin{equation}\label{nosubsum}
    \sum_{i\in I} f_i \ne 0 \qquad\text{for each proper index set~}
    I\subsetneq\{0,\dots,n\}.
\end{equation}
If $\mathbf{F}$ has characteristic zero, let \hbox{$c=1,$} and if
$\mathbf{F}$ has positive characteristic $p,$ let $c=p^{s-1},$
where $s$ is the index of independence for the $f_i.$
Then, there exist integers $\bar a$ and $b$ with
$$
    1 \le \bar a \le c \sum_{i=1}^{\bar k -1}(d-i)
    \qquad\text{and}\qquad \bar a \le b
$$
such that for $r\ge1,$
\begin{equation}\label{abctrprod}
    \max_{0\le j \le n}\log|f_j|_r \le
    N^{(\bar a)}_F(0,r) - b \log r +O(1),
\end{equation}
where $F=f_0\cdots f_n.$
\end{theorem}

We remark that given explicit $f_i,$ the constants $a,$ $b,$ $\sigma,$
and $\bar a$ in Theorems~\ref{abcsum} and~\ref{abcprod}
can be determined explicitly in terms of non-vanishing generalized
Wronskians of subsets of the $f_i,$
as will be evident from the proof; see~(\ref{a}), (\ref{b}),
(\ref{abar}), and (\ref{sigma}).

Trivial examples
in positive characteristic show that the dependence on the
index of independence $s$ cannot be removed.

In characteristic zero and one variable, Theorems~\ref{abcsum}
and~\ref{abcprod} are due to Hu and Yang \cite{HuYangBB}.  Below, we
will give their proof, which in turn closely follows Brownawell and
Masser \cite{BrownawellMasser}, and simply observe that it works,
given the proper set-up, just as well for several variables and in
positive characteristic. That it is natural to express the upper
bounds on $a$ and $\bar a$ in terms of $d$ is an observation of
Zannier \cite{Zannier}. For polynomials of several variables,
De~Bondt \cite{deBondt} gave an alternative Wronskian based proof
that first reduces to the case of one variable polynomials by
generic specialization but then proves the one variable case by
introducing extra variables to force linear independence, thereby
avoiding Lemma~\ref{bm}.

Before giving the proof of Theorems~\ref{abcsum} and~\ref{abcprod},
we discuss how to
derive from it various other existing results in the literature.
First observe that in characteristic zero, $b\ge a(a+1)/2$ and
so~(\ref{abctrsum}) implies
$$
    \max_{0\le j \le n}\log|f_j|_r \le
    \sum_{j=0}^n N^{(a)}_{f_j}(0,r) - \frac{a(a+1)}{2}\log r + O(1),
$$
which specializing to complex polynomials gives us
\begin{corollary}[{\cite[Th.~2.1~(4)]{deBondt}}]\label{deBondtFour}
Let $f_0,\dots,f_n$ be polynomials of several complex variables
satisfying the hypothesis of the theorem. Then,
$$
    \max \deg f_j \le \sum_{j=0}^n r_a(f_j) - \frac{a(a+1)}{2},
$$
where $r_a(f_j)=\deg\gcd(f_j,R(f_j)^a)$ and $a$ is as in the theorem.
\end{corollary}

Incorporating an idea of Bayat and Teimoori \cite{BayatTeimoori}
as in De~Bondt \cite{deBondt}, we get

\begin{corollary}\label{deBondtFive}
Assume $\mathbf{F}$ has characteristic zero,
let \hbox{$f_0+\dots+f_n=0$} be as in Theorem~\ref{abcsum} and let $d$
be the dimension of the $\mathbf{F}$ vector space spanned by
the $f_j.$
Let $C$ be the number of $f_j$ which are constant functions.
For any $A$ with $d\le A \le n-C,$
$$
    \max_{0\le j \le n}\log|f_j|_r \le
    A\left(\sum_{j=0}^n N^{(1)}_{f_j}(0,r)-\frac{A+1}{2}\log r\right)
    +O(1)
$$
for $r\ge1.$
\end{corollary}

De~Bondt gives examples that show that $A$ cannot be made smaller
than $d$ in Corollary~\ref{deBondtFive}.

In the case of polynomials, Corollary~\ref{deBondtFive}
gives \cite[Th.~2.1~(5)]{deBondt},
which implies \cite[Th.~5]{BayatTeimoori} as explained in
\cite{deBondt}. As remarked by De~Bondt, the proof of
\cite[Th.~5]{BayatTeimoori} given by Bayat and Teimoori
in \cite{BayatTeimoori} is not correct
for polynomials of several variables because their Lemma~4 is easily
seen to be false for several variable polynomials. However, arguing
as in \cite{deBondt}, their Theorem~5 is correct, even for several
variables. Of course, this also recovers the result
of Shapiro and Sparer \cite{ShapiroSparer}.

\begin{proof}[Proof of Corollary~\ref{deBondtFive}] Clearly
(by Proposition~\ref{countingfactorsmaller}),
$$
    N^{(a)}_{f_j}(0,r) \le a N^{(1)}_{f_j}(0,r),
$$
so from~(\ref{abctrsum}), we get
$$
    \max_{0\le j \le n}\log|f_j|_r \le
    a\left(\sum_{j=0}^n N^{(1)}_{f_j}(0,r) - \frac{a+1}{2}\log r\right)
+ O(1)
$$
for $r\ge1.$  The observation of Bayat and Teimoori is that the
expression on the right is quadratic in $a$ and hence increasing
in $a$ provided
$$
    a + \frac{1}{2} \le \frac{1}{\log r}\sum_{j=0}^n N^{(1)}_{f_j}(0,r).
$$
Because
$$
    N_{f_j}^{(1)}(0,r)\ge \log r + O(1)
$$
for any non-constant $f_j$ and because there are $n+1-C$
non-constant $f_j,$ we have that
$$
    \sum_{j=0}^n N^{(1)}_{f_j}(0,r) \ge (n+1-C)\log r + O(1)
    \ge (n-C)\log r
$$
for $r$ sufficiently large, and so
we can increase $a$ up to $n-C$ and the inequality
will hold for all sufficiently large $r.$  We can then adjust the
$O(1)$ term to make the inequality hold for $r\ge1.$
\end{proof}

We now digresss a little bit to discuss one slightly subtle
difference between entire functions and polynomials.
The astute reader will notice that we assumed no vanishing subsums,
\textit{i.e.,}~(\ref{nosubsum}), in Theorem~\ref{abcprod}, whereas
De~Bondt assumed the weaker hypothesis~(\ref{subsum}) in both versions
of his ABC theorems.  In the case of complex polynomials, the inequality
in De~Bondt's work that corresponds to our inequality~(\ref{abctrprod})
is
$$
    \max \deg f_j \le \deg\gcd(F,R(F)^{\bar a})-b
    \le \deg\gcd(F,R(F)^{\bar a})-\bar a,
$$
and this holds even if the hypothesis of no vanishing subsums~(\ref{nosubsum})
is weakened to~(\ref{subsum}).  The reason for this is that in the
case of polynomials, one of the polynomials $f_{j_0}$ has maximal degree.
Thus, one need only consider a minimal index set $I$ such that $j_0$
is contained in $I$ and such that
$$
    \sum_{j\in I}f_j=0.
$$
However, in the case of entire functions, it need not be the case
that there is a fixed index $j_0$ such that for all $r$ sufficiently
large
$$
    |f_{j_0}|_r = \max_{0\le j \le n} |f_j|_r.
$$
Thus, if one replaces the hypothesis~(\ref{nosubsum})
in Theorem~\ref{abcprod} with the weaker hypothesis~(\ref{subsum}),
it follows from~(\ref{abctrprod}) and the fact that
for a positive integer $\ell,$ for $r\ge1$ and $F$ and $G$ entire functions
$$
    \max\{N_F^{(\ell)}(0,r),N_G^{(\ell)}(0,r)\}\le N_{FG}^{(\ell)}(0,r)
$$
that for each index $j$ in $\{0,\dots,n\},$ there exist integers
$\bar a_j$ and $b_j$ with
$$
    1 \le \bar a_j \le c\sum_{i=1}^k(d-i)
    \qquad\text{and}\qquad \bar a_j \le b_j
$$
such that for $r\ge1,$
$$
    \log |f_j|_r \le N_{F}^{(\bar a_j)}(0,r) - b_j \log r + O(1),
$$
from which it follows that
$$
    \max_{0\le j \le n} \log |f_j|_r \le
    N_{F}^{(\max \bar a_j)}(0,r) - (\min b_j) \log r +O(1).
$$
Of course, it need not be that $\min b_j \ge \max \bar a_j,$ and thus
when subsums of the $f_j$ may vanish,
it is not clear for entire functions whether one can choose the same
constant at which multiplicities are truncated when counting the zeros of
$F$ as the coefficient in front of $-\log r.$  Whether that can be done
is a somewhat interesting question, because if it can be done, then
the proof cannot be a straightforward generalization of the
existing polynomial proof.  If it cannot be done, then this would be
an example where a polynomial inequality does not completely generalize
to an analogous inequality for entire functions.

\begin{corollary}\label{abcsf}
With hypotheses and notation as in Theorem~\ref{abcprod}, we have
for $r\ge1$ and any
$$
    \bar A \ge c\sum_{i=1}^{\bar k -1}(d-i),
$$
that
$$
    \max_{0\le j \le n}\log|f_j|_r \le
    \bar A \left(N^{(1)}_F(0,r)-\log r\right)+O(1).
$$
Moreover, the above inequality remains valid if the
hypothesis~(\ref{nosubsum}) is weakened to hypothesis~(\ref{subsum}).
\end{corollary}

\begin{remark*}The $O(1)$ term may depend on $\bar A.$
\end{remark*}

\begin{proof}
Under the hypothesis~(\ref{nosubsum}),
we have
$$
    \max_{0\le j \le n}\log|f_j|_r \le
    N^{(\bar a)}_F(0,r) - \bar a \log r +O(1),
$$
which follows immediately from~(\ref{abctrprod})
because $\bar a \le b.$ Because
$$
    N^{(\bar a)}_F(0,r)\le \bar a N^{(1)}_F(0,r)
$$
and because
$$
    N^{(1)}_F(0,r)-\log r
$$
is bounded below for $r\ge1$ since $F$ is non-constant, we have
$$
    N^{(\bar a)}_F(0,r) - \bar a \log r
    \le \bar a \left(N^{(1)}_F(0,r)-\log r\right)
    \le \bar A \left(N^{(1)}_F(0,r)-\log r\right)+O(1),
$$
which gives the corallary when there are no vanishing subsums.
However, even if there are vanishing subsums, the functions can be
grouped into vanishing subsums
$$
    \sum_{i\in I} f_i = 0
$$
with no vanishing sub-subsums.
Any vanishing subsum consisting of all constants can be thrown out.
Letting
$$
    F_I=\prod_{i \in I} f_i,
$$
we have
$$
    \max_{j \in I} \log |f_j|_r \le \bar A\left(N^{(1)}_{F_I}(0,r)-\log r
    \right)+O(1).
$$
Because $\bar A$ was chosen independent of $I$ and because
$$
    \max_{I} N^{(1)}_{F_I}(0,r) \le N^{(1)}_F(0,r) \qquad\text{for~}r\ge1,
$$
the corollary follows in general.
\end{proof}

If $\mathbf{F}$ has characteristic zero and in the case
of polynomials when $k=n,$ Corollary~\ref{abcsf} is
\cite[Th.~2.2~(7)]{deBondt}. When $k=3,$
then we can take $\bar A= 2n-3$
and so we recover the main result
of Quang and Tuan in \cite{QuangTuan}:
\begin{equation}\label{BB}
    \max_{0\le j \le n} \deg f_j \le (2n-3)[\deg R(F)-1]
\end{equation}
if $\gcd(f_{i_1},f_{i_2},f_{i_3})=1$ for all triples
of indices $i_1<i_2<i_3.$ Note that Quang and Tuan
neglected the necessary hypothesis that the functions in
any vanishing subsum be relatively prime, \textit{i.e.,}
hypothesis~(\ref{subsum}).
We also remark that Browkin and Brzezinski \cite{BrowkinBrzezinski}
conjectured that~(\ref{BB}) remains true
for one variable polynomials in characteristic zero
if the $\gcd$ hypothesis
is relaxed to $\gcd(f_0,\ldots,f_n)=1$ and no vanishing subsums.
This conjecture
seems to be out of reach of the current Wronskian based proofs.

Fundamental to the proof of Theorems~\ref{abcsum}
and~\ref{abcprod} are the following lemmas
about generalized Wronskians.

\begin{lemma}\label{divW}
Let $f_0,\ldots,f_{n-1}$ be entire functions in $\mathcal{E}_m.$
Let
$$
    \gamma^1=(\gamma_1^1,\dots,\gamma_m^1),\dots,
    \gamma^{n-1}=(\gamma_1^{n-1},\dots,\gamma_m^{n-1})
$$
be multi-indices
with $|\gamma^1|\le\ldots\le|\gamma|^{n-1}$ such that
the associated generalized Wronskian $W$ does not vanish identically.
Let $\gamma^0=(0,\ldots,0).$
If $P$ is an irreducible element which divides $f_i$ with
multiplicity $e>|\gamma^{n-1}|,$ then $P$ divides $W$
with multiplicity at least $e-|\gamma^{n-1}|.$
Moreover, if \hbox{$\mathrm{char~}\mathbf{F}=p>0$}
and if $p^t$ divides $e$ and
$$
    p^t>\max\{\gamma_i^j :
    1\le j \le n-1 \text{~and~}
    1\le i \le m\},
$$
then $P^e$ divides $W.$
\end{lemma}

\begin{remark*} Note that
\hbox{$\max\{\gamma_i^j :
1\le j \le n-1 \text{~and~}
1\le i \le m\}\le|\gamma^{n-1}|.$}
\end{remark*}

\begin{proof} This follows immediately from Proposition~\ref{divDgamma}.
\end{proof}

\begin{lemma}\label{wronsktrunc}
Let $f_0,\ldots,f_{n-1}$ be entire functions in $\mathcal{E}_m.$
Let $\gamma^1,\dots,\gamma^{n-1}$ be multi-indices
with $|\gamma^1|\le\ldots\le|\gamma|^{n-1}$ such that
the associated generalized Wronskian $W$ does not vanish identically.
Let $\gamma^0=(0,\ldots,0).$
Let $F=f_0\cdots f_{n-1}.$
Let $k\ge2$ be the smallest integer such that for
every $k$ distinct indices $i_1,\dots,i_k$ in
$\{0,\dots,n-1\}$ one has $\gcd(f_{i_1},\dots,f_{i_k})=1,$
or if no such $k$ exists, let $k=n+1.$
Let
$$
    \ell=\sum_{i=1}^{k-1}|\gamma^{n-i}|.
$$
If $P$ is an irreducible element which divides $F$ with multiplicity
$e>\ell,$
then $P$ divides $W$ with multiplicity at least $e-\ell.$
\end{lemma}

\begin{remark*} Note that in positive characteristic $p,$
even if $P$ divides $F$ with multiplicity a large multiple of $p,$
it does not necessarily mean that $P$ divides $W$ to the same multiplicity.
This is because the powers of $P$ may be split among the different $f_i,$
so $P$ need not divide any of the $f_i$ with multiplicity divisible by $p.$
\end{remark*}

\begin{proof}
By the hypothesis of the lemma, we may assume without loss
of generality by re-ordering the indices if necessary
that $P$ does not divide $f_j$ for $j\ge k-1.$  Let $e_i\ge0$
for $i=0,\dots,k-2$
be the multiplicities with which $P$ divides the $f_i,$ and
assume $e_0\ge \ldots \ge e_{k-2}.$ Then, by Proposition~\ref{divDgamma},
$P$ divides $W$ with multiplicity at least
$$
    \sum_{i=0}^{k-2}\max\{0,e_i-|\gamma^{n-i-1}|\}
    \ge \sum_{i=0}^{k-2}(e_i-|\gamma^{n-i-1}|)=e-\ell. \qedhere
$$
\end{proof}

\begin{proof}[Proof of Theorems~\ref{abcsum} and~\ref{abcprod}.]
If $\mathbf{F}$ has positive characteristic, let $G_i$ be as in the
statement of Theorem~\ref{abcsum}.  If $\mathbf{F}$ has
characteristic zero, let \hbox{$G_i=\gcd(f_i,S(f_i)^a).$} Then,
$G_i$ divides $\gcd(f_i,S(f_i)^a),$ and thus
$$
    N_{G_i}(0,r) \le N_{f_i}^{(a)}(0,r)
$$
for $r\ge1$ by Proposition~\ref{countingfactorsmaller}
and the definition of truncated counting functions.
Therefore, inequality~(\ref{abctrsum}) follows from
inequality~(\ref{abctrsumcharp}).  Hence, to prove Theorem~\ref{abcsum},
it suffices to prove~(\ref{abctrsumcharp}), where in characteristic
zero we interpret~(\ref{abctrsumcharp}) with $G_i$ as defined here in
the proof.

We next observe that it suffices to prove Theorem~\ref{abcsum}
assuming there are no vanishing subsums.  Indeed, if there are
vanishing subsums, simply group the $f_i$ into vanishing subsums
with no vanishing sub-subsums, and note that these subsums
still satisfy all the hypotheses of the theorem.
Summing~(\ref{abctrsumcharp})
over each minimal vanishing subsum clearly results
in~(\ref{abctrsumcharp}) for the general case.

Note also that if $\bar k=2$ in Theorem~\ref{abcprod}, then the
$f_j$ are pairwise relatively prime, in which case
$$
    N_F^{(\ell)}(0,r)=\sum_{j=0}^n N_{f_j}^{(\ell)}(0,r),
$$
and so Theorem~\ref{abcprod} follows from Theorem~\ref{abcsum}.
Thus, we henceforth assume~(\ref{nosubsum}) as we prove both theorems.

Consider the
$\mathbf{F}$-linear span of $f_0,\ldots,f_n$ as a $\mathbf{F}$-vector space,
and partition
$$
    \{0,\ldots,n\}=I_0\cup\ldots\cup I_{u-1},
    \qquad\textnormal{with~} J_j\in I_j \textnormal{~for~}j=0,\ldots,u-2
$$
as in Lemma~\ref{bm}. Let $J_{-1}=\emptyset.$
Also, without loss of generality, assume $0\in I_0.$
Let $n_j$ be the cardinality of $I_j$
for $j=0,\dots,u-1.$ Note that $n_0\le d+1$ by the minimality of $I_0$
and that $n_j\le d$ for $j\ge1$ by the minimality of
$I_j\cup J_{j-1}.$
Set
$$
    \gamma^{0,0}=\gamma^{1,0}=\ldots=\gamma^{u-1,0}=
    \gamma^0=\gamma^1=\ldots=\gamma^{u-1}=(0,\ldots,0).
$$
The $f_i$ for $i$ in $I_0\setminus\{0\}$ are linearly independent
over $\mathbf{F}.$ Therefore by Lemma~\ref{linindep} and
Theorem~\ref{gwronsk},
there exist multi-indices $\gamma^{0,1},\ldots,\gamma^{0,n_0-2}$
such that the generalized Wronskian $W_0$ formed by the $f_i$ with
respect to these multi-indices for
$i$ in $I_0\setminus\{0\}$ is not identically zero,
and moreover,
$$
    |\gamma^{0,i}|\le |\gamma^{0,i-1}|+c\le ci.
$$
Note that
if $n_0=2,$ we simply let $W_0$ be the $f_i$ for the unique
index $i$ in $I_0$ different from $0.$
Similarly, for $j=1,\ldots,u-1,$ there exist multi-indices
$\gamma^{j,1},\ldots,\gamma^{j,n_j-1}$ such that the generalized
Wronskian $W_j$ formed by the $f_i$ with respect to these multi-indices
for $i$ in $I_j$ is not identically zero
and
$$
    |\gamma^{j,i}|\le |\gamma^{j,i-1}|+c\le ci.
$$
The total number of multi-indices
we get this way is
$$
    n_0-2+n_1-1+\ldots+n_{u-1}-1=n-u.
$$
Write these multi-indices as $\gamma^u,\ldots,\gamma^{n-1}$ with
$$
    |\gamma^u|\le\ldots\le|\gamma^{n-1}|\le c(d-1),
$$
and note that
\begin{equation}\label{upone}
    |\gamma^i| \le |\gamma^{i-1}| + c.
\end{equation}
For each minimal index set $I_0,$ $J_0\cup I_1,$ \dots $J_{u-2}\cup I_{u-1},$
there is a linear dependence relation:
$$
    \sum_{i=0}^{n} c_{j,i}f_i=0,
$$
with $c_{j,i}$ non-zero elements of $\mathbf{F}$ when $i$
is in $J_{j-1}\cup I_j$ and $0$ otherwise. Of course this also
gives rise to the linear equations
$$
    \sum_{i=0}^n c_{0,i}
    D^{\gamma^{0,q}}\!f_i=0, \qquad
    q=0,\ldots,n_0-2
$$
and
$$
    \sum_{i=0}^n c_{j,i}
    D^{\gamma^{j,q}}\!f_i=0, \qquad
    q=0,\ldots,n_j-1, j=1,\ldots,u-1.
$$
Let $M$ be the $n\times(n+1)$-matrix
whose entries are $c_{j,i}D^{\gamma^{j,q}}\!f_i,$
where the columns are indexed by $i$ and the rows are indexed by
$j$ and $q.$  Note that the sum of each row of $M$ is zero.
Let $\Delta_i$ denote the determinant of the matrix $M$ with the $i$-th
column deleted.  Because the $i$-th column is the negative
of the  sum of the other columns,
$\Delta_i=\pm \Delta_j.$ From the block nature of $M,$
\begin{equation}\label{Dprod}
    \Delta_{0} = C_{0}W_0\cdots W_{u-1},
\end{equation}
where $C_{0}$ is a constant obtained by multiplying the
appropriate $c_{j,i}$'s, and hence is non-zero. Thus, $\Delta_i$
is non-zero for all $i,$ and up to a constant
is the product of the generalized Wronskians $W_j.$

Define
\begin{equation}\label{a}
    a = |\gamma^{n-1}|\le c(d-1),
\end{equation}
\begin{equation}\label{b}
    b=\sum_{i=u}^{n-1}|\gamma^i|\ge
    \sum_{i=u}^{n-2}\max\{1,|\gamma^{n-1}|-ci\}
    \ge a\left\lceil\frac{a}{c}\right\rceil-
    \frac{\displaystyle\left\lceil\frac{a}{c}\right\rceil
    \left(\left\lceil\frac{a}{c}\right\rceil-1\right)}{2}c\ge a
\end{equation}
and
\begin{equation}\label{abar}
    \bar a = \sum_{i=1}^{\bar k-1}|\gamma^{n-i}|
    \le c \sum_{i=1}^{\bar k -1}(d-i),
\end{equation}
where the first inequality in~(\ref{b}) follows from~(\ref{upone}).
Note also that
\begin{equation}\label{bbiggerabarbiggera}
    b\ge\bar a \ge a \ge 1,
\end{equation}
where $1\le a$ follows from the fact that not all the Wronskians $W_j$
can be $1\times1,$ for if they were, we would have
either all the $f_i$ constant or $\gcd(f_0,\ldots,f_n)\ne1.$

If $\mathbf{F}$ has characteristic $p>0,$ let $\sigma$ be the largest
integer such that
\begin{equation}\label{sigma}
    p^\sigma \le \max\{\gamma^i_j :
    u\le i \le n-1 \text{~and~}
    1\le j \le m\}\le a
    \text{~where~}
    \gamma^i=(\gamma^i_1,\dots,\gamma^i_m).
\end{equation}

We claim that $F$ divides
$$
    \Delta_0 \prod_{i=0}^n G_i.
$$
Indeed, suppose that $P$ is an irreducible element of
$\mathcal{E}_m$ which divides $f_i$ with exact multiplicity $e.$
If $\mathbf{F}$ has charactersitic zero, then
$P^{\min\{e,a\}}$ divides $G_i$ by Proposition~\ref{radical},
and if $e>a,$ then $P^{e-a}$ divides $W_j,$
where $i\in I_j,$ by Lemma~\ref{divW} and~(\ref{a}).
Now suppose $\mathrm{char~}\mathbf{F}=p$ and $p^v$ is the largest
power of $p$ dividing $e.$
If $p^v\le p^\sigma \le a,$ then $P^{\min\{e,a\}}$ divides $G_i$
by Propositon~\ref{pradical} and in the case $e>a,$
we have that $P^{e-a}$ divides $W_j,$
where $i\in I_j,$  again by Lemma~\ref{divW} and~(\ref{a}).
If $p^v>p^\sigma,$ then $P^e$
divides $\Delta_0$ by Lemma~\ref{divW} and~(\ref{sigma}).
Thus in all cases,
$P^e$ divides $W_jG_i,$ which divides $\Delta_0G_i$ by~(\ref{Dprod}).
Therefore,
\begin{equation}\label{Fsum}
    \log|F|_r\le\log|\Delta_0|_r + \sum_{i=0}^n \log|G_i|_r
    +O(1)
\end{equation}
by Corollary~\ref{factorsmaller}.

Let
$$
    F_0=\prod_{i\in I_0\setminus \{0\}} f_i, \qquad\textnormal{and}\qquad
    F_j=\prod_{i\in I_j} f_i \qquad j=1,\ldots,u-1.
$$
For each $j=0,\dots,u-1,$ the quotient $W_j/F_j$ is a determinant
of a matrix consisting of logarithmic derivatives, and so by
Lemma~\ref{HasseLDL},
$$
    \left|\frac{W_0}{F_0}\right|_r \le -\left(\sum_{q=0}^{n_0-2}
    |\gamma^{0,q}|\right)\log r
$$
and
$$
    \left|\frac{W_j}{F_j}\right|_r \le -\left(\sum_{q=0}^{n_j-1}
    |\gamma^{0,q}|\right)\log r \qquad\text{for~}j=1,\dots,u-1.
$$
Then,
$$
    \frac{f_{0}\Delta_{0}}{F}=C_{0}\frac{W_0}{F_0}\cdots
    \frac{W_{u-1}}{F_{u-1}}
$$
and hence,
$$
    \log|f_{0}|_r+\log|\Delta_{0}|_r-\log|F|_r \le -b\log r + O(1)
$$
by Lemma~\ref{HasseLDL} and~(\ref{b}).
Similarly, if $i$ is any index in $I_0,$
then we can write $W_0$ as a determinant involving $D^{\gamma^{0,q}}\!f_\ell$
for $\ell\ne i,$ and so
$$
    \frac{f_iW_0}{f_{0}F_0}
$$
is also a sum of products of logarithmic derivatives, and hence
$$
    \log\left|\frac{f_iW_0}{f_{0}F_0}\right|_r \le
    -\left(\sum_{q=0}^{n_0-2}|\gamma^{0,q}|\right)\log r
$$
too. Thus
$$
    \log|f_i|_r+\log|\Delta_{0}|_r-\log|F|_r \le -b\log r + O(1)
$$
for all $i$ in $I_0.$ Now let $i$ be in $I_1$
and let $j$ be in $J_0.$  Then, by a similar argument,
$$
    \frac{f_i\Delta_0}{F}=C_0\cdot\frac{f_jW_0}{f_0F_0}\cdot
    \frac{f_iW_1}{f_jF_1}\cdot\frac{W_2}{F_2}\cdots\frac{W_{u-1}}{F_{u-1}},
$$
and
$$
    \frac{f_jW_0}{f_{0}F_0}\qquad\textnormal{and}\qquad
    \frac{f_iW_1}{f_jF_1}
$$
are both sums and products of logarithmic derivatives. Hence
$$
    \log|f_i|_r+\log|\Delta_{0}|_r-\log|F|_r \le -b\log r + O(1).
$$
Continuing like so, we find that for all $i,$
$$
    \log|f_i|_r+\log|\Delta_{0}|_r-\log|F|_r \le -b\log r + O(1),
$$
whence
\begin{equation}\label{mainineq}
    \max_{0\le i \le n}\log|f_i|_r \le \log |F|_r - \log|\Delta_{0}|_r
    -b\log r + O(1).
\end{equation}
Combining this with~(\ref{Fsum}), we get
$$
    \max_{0\le i \le n}\log|f_i|_r \le \sum_{i=0}^n
    \log|G_i|_r
    -b\log r + O(1).
$$
Using the Poisson-Jensen-Green type formula~(\ref{poisson})
and the definition of counting functions, this can also be
written, for $r\ge1,$ as
$$
    \max_{0\le i \le n}\log|f_i|_r \le \sum_{i=0}^n
    N_{G_i}(0,r)
    -b\log r + O(1),
$$
which is precisely~(\ref{abctrsumcharp}).

We now show~(\ref{abctrprod}).
Let $P$ be an irreducible element of $\mathcal{E}_m$ that divides
$F.$  By the hypotheses of the theorem, there is at least one $f_i$
such that $P$ does not divide $f_i.$  Because, as above, we can write
$\Delta_{0}$ as a product of Wronskians not involving $f_i,$
we can use
Lemma~\ref{wronsktrunc} to conclude that $P$
divides $\Delta_{0}$ with multiplicity at least $e-\bar a$ by~(\ref{abar}).
Thus, $F/\gcd(F,\Delta_{0})$ divides
$\gcd(F,S(F)^{\bar a}).$
Hence, by Proposition~\ref{factorsmaller}, the Poisson-Jensen-Green
type formula~(\ref{poisson}),
and the definition of truncated counting functions, for $r\ge1,$
$$
    \log |F|_r - \log|\Delta_{0}|_r \le |\gcd(F,S(F)^{\bar a})|_r +O(1)
    = N^{({\bar a})}_F(0,r)+O(1).
$$
Combining this with~(\ref{mainineq}), we have for $r\ge1,$
$$
    \max_{0\le i \le n}\log|f_i|_r \le N^{({\bar a})}_F(0,r)
    -b\log r + O(1),
$$
which is~(\ref{abctrprod}).
\end{proof}


\end{document}